\documentclass[10pt,reqno]{amsart}
\usepackage{amsmath,amscd,amsfonts,amsthm,amssymb,latexsym,mathrsfs,graphicx,url,color}

\usepackage{hyperref}
\hypersetup{
    bookmarks=true,         
    unicode=false,          
    pdftoolbar=true,        
    pdfmenubar=true,        
    pdffitwindow=false,     
    pdfstartview={FitH},    
    pdftitle={My title},    
    pdfauthor={Author},     
    pdfsubject={Subject},   
    pdfcreator={Creator},   
    pdfproducer={Producer}, 
    pdfkeywords={keyword1} {key2} {key3}, 
    pdfnewwindow=true,      
    colorlinks=true,       
    linkcolor=black,          
    citecolor=blue,        
    filecolor=blue,      
    urlcolor=blue      
    }

\usepackage{hyperref, caption, subcaption}

\theoremstyle{plain}
   \newtheorem{theorem}{Theorem}[section]
   \newtheorem{proposition}[theorem]{Proposition}
   \newtheorem{lemma}[theorem]{Lemma}
   \newtheorem{corollary}[theorem]{Corollary}
   
   \newtheorem{conjecture}[theorem]{Conjecture}
   
\theoremstyle{definition}
   \newtheorem{definition}{Definition}[section]
   \newtheorem{example}{Example}[section]

\theoremstyle{remark}
   \newtheorem{remark}[theorem]{Remark}
   
\numberwithin{equation}{section}

\def\kk{\kern.2ex\mbox{\raise.5ex\hbox{{\rule{.35em}{.12ex}}}}\kern.2ex}

\newcommand{\NN}{\mathbb{N}}
\newcommand{\lma}{\lambda_{\rm max}}
\newcommand{\lmi}{\lambda_{\rm min}}

\newcommand{\ww}{\mathbf{w}}
\newcommand{\ee}{\mathbf{e}}

\newcommand{\xx}{\mathbf{x}}
\newcommand{\yy}{\mathbf{y}}
\newcommand{\uu}{\mathbf{u}}

\newcommand{\vv}{\mathbf{v}}

\newcommand{\one}{\mathbf{1}}

\newcommand{\MM}{\mathcal{M}}

\newcommand{\BB}{\mathcal{B}}
\newcommand{\VV}{\mathcal{V}}

\newcommand{\FF}{\mathcal{F}}

\newcommand{\RR}{\mathbb{R}}
\newcommand{\CC}{\mathbb{C}}

\renewcommand{\Im}{{\rm Im}}

\newcommand{\rk}{{\rm rk}}

\title[Non--representable hyperbolic matroids]{Non--representable hyperbolic matroids} 

\author{Nima Amini}
\author{Petter Br\"and\'en}
\thanks{The second author is a Wallenberg Academy Fellow 
  supported by a grant from the Knut and Alice Wallenberg
  Foundation, and the Swedish Research Council (VR)}

\address{Department of Mathematics, Royal Institute of Technology, SE-100 44 Stockholm,
Sweden}
\email{namini@kth.se, pbranden@kth.se}

\begin{document}
\begin{abstract}
The generalized Lax conjecture asserts that each hyperbolicity cone is a linear slice of the cone of positive semidefinite matrices. Hyperbolic polynomials give rise to a class of (hyperbolic) matroids which properly contains the class of matroids representable over the complex numbers. This connection was used by the second author to construct counterexamples to algebraic (stronger) versions of the generalized Lax conjecture by considering a non--representable hyperbolic matroid. The V\'amos matroid and a generalization of it are, prior to this work, the only known instances of non--representable hyperbolic matroids. 

We prove that the Non--Pappus and Non--Desargues matroids are non-representable hyperbolic matroids by exploiting a connection between Euclidean Jordan algebras and projective geometries. We further identify a large class of hyperbolic matroids which contains the V\'amos matroid and the generalized V\'amos matroids recently studied by Burton, Vinzant and Youm. This proves a conjecture of Burton \emph{et al}. We also prove that many of the matroids considered here are non--representable. The proof of hyperbolicity for the matroids  in the class depends on proving nonnegativity of certain symmetric polynomials. In particular we generalize and strengthen several inequalities in the literature, such as the Laguerre--Tur\'an inequality and Jensen's inequality. Finally we explore consequences to algebraic versions of the generalized Lax conjecture. 
\end{abstract}
\maketitle
\tableofcontents

\renewcommand\theenumi{\roman{enumi}}
\thispagestyle{empty}
\newpage
\section{Introduction}
Although hyperbolic polynomials have their  origin  in PDE theory, they have during recent years  been studied in diverse areas such as 
control theory, optimization, real algebraic geometry, probability theory, computer science and combinatorics, see  \cite{Pem,R,Vin,Wag} and the references therein. To each hyperbolic polynomial is associated a closed convex (hyperbolicity) cone. Over the past 20 years methods have been developed to do optimization over hyperbolicity cones, which generalize semidefinite programming. A problem that has received considerable interest is the generalized Lax conjecture which asserts that each hyperbolicity cone is a linear slice of the cone of positive semidefinite matrices (of some size). Hence if the generalized Lax conjecture is true then hyperbolic programming is the same as semidefinite programming. 

Choe \emph{et al.} \cite{COSW} and Gurvits \cite{Gur} proved that hyperbolic polynomials give rise to a class of matroids,  called hyperbolic matroids or matroids with the weak half--plane property. The class of hyperbolic matroids properly  contains the class of matroids which are representable over the complex numbers, see \cite{COSW, WW}. This fact   was used by the second author \cite{B} to construct counterexamples to algebraic (stronger) versions of the generalized Lax conjecture. To better understand, and to identify potential counterexamples to the generalized Lax conjecture, it is therefore of interest to study hyperbolic matroids which are not representable over $\mathbb{C}$, or even better not representable over any (skew) field. However previous to this work essentially just two such matroids were known: The V\'amos matroid $V_8$ \cite{WW} and a generalization $V_{10}$ \cite{BVY}. In this paper we first show that the Non-Pappus and Non-Desargues matroids are hyperbolic (but not representable over any field) by utilizing  a known connection between hyperbolic polynomials and Euclidean Jordan algebras. Then we construct a family of hyperbolic matroids, Theorem \ref{genvamoshyp}, which are parametrized by uniform hypergraphs, and prove  that many of these matroids fail to be representable over any field, and more generally over any modular lattice. The proof of the main result is involved and uses several ingredients. In order to prove that the polynomials coming from our family of matroids are hyperbolic we need to prove that certain symmetric polynomials are nonnegative. The results obtained generalize and strengthen several inequalities in the literature, such as the Laguerre--Tur\'an inequality and Jensen's inequality. Finally we explore some consequences to algebraic versions of the generalized Lax conjecture.

\section{Hyperbolic and stable polynomials}

A homogeneous polynomial $h(\xx) \in \RR[x_1, \ldots, x_n]$ is \emph{hyperbolic} with respect to a vector 
$\ee \in \RR^n$ if $h(\ee) \neq 0$, and if for all $\xx \in \RR^n$ the univariate polynomial $t \mapsto h(t\ee -\xx)$ has only real zeros. Note that if $h$ is a hyperbolic polynomial of degree $d$, then we may write
\begin{align*}
\displaystyle h(t\mathbf{e} - \mathbf{x}) = h(\mathbf{e}) \prod_{j=1}^d (t - \lambda_j(\mathbf{x})), 
\end{align*} \noindent
where $$\lma(\xx)=\lambda_1(\mathbf{x}) \geq \cdots \geq \lambda_d(\mathbf{x})=\lmi(\xx)$$ are called the \emph{eigenvalues} of $\mathbf{x}$ (with respect to $\mathbf{e}$). The \emph{hyperbolicity cone} of $h$ with respect to $\mathbf{e}$ is the set $\Lambda_{+}(h, \mathbf{e}) = \{ \mathbf{x} \in \mathbb{R}^n : \lmi(\mathbf{x}) \geq 0 \}$. We usually abbreviate and write $\Lambda_{+}$ if there is no risk for confusion. We denote by $\Lambda_{++}$ the interior  $\Lambda_{+}$. 
\begin{example}\label{detta}
An important example of a hyperbolic polynomial is $\det(X)$, where $X = (x_{ij})_{i,j =1}^n$ is a matrix of variables where we impose $x_{ij} = x_{ji}$. Note that $t \mapsto \det(tI - X)$ where $I = \text{diag}(1,\dots, 1)$, is the characteristic polynomial of a symmetric matrix so it has only real zeros. Hence $\det(X)$ is a hyperbolic polynomial with respect to $I$, and its hyperbolicity cone is the cone of positive semidefinite matrices. 

The real linear space of complex hermitian matrices of size $n$ is  parametrized by matrices $X$ in $n^2$ variables, and as above it follows that 
$\det(X)$ is a hyperbolic polynomial. 
\end{example} \noindent

The next theorem which follows (see \cite{LPR}) from a theorem of Helton and Vinnikov \cite{HV} proved the Lax conjecture (after Peter Lax $1958$ \cite{Lax}). 
\begin{theorem}\label{heltonvinnikov}
Suppose that $h(x,y,z)$ is of degree $d$ and hyperbolic with respect to $e = (e_1,e_2,e_3)^T$. Suppose further that $h$ is normalized such that $h(e) = 1$. Then there are symmetric $d \times d$ matrices $A, B, C$ such that $e_1A+e_2B+ e_3C = I$ and
\begin{align*}
\displaystyle h(x,y,z) = \det(xA + yB + zC).
\end{align*}
\end{theorem} \noindent
\begin{remark}
The exact analogue of the Helton-Vinnikov theorem fails for $n>3$ variables. This may be seen by comparing dimensions. The space of polynomials on $\mathbb{R}^n$ of the form $\det(x_1A_1 + \cdots x_nA_n)$ with $A_i$ symmetric for $1 \leq i \leq n$, has dimension at most $n \binom{d}{2}$ whereas the space of hyperbolic polynomials on $\mathbb{R}^n$ has dimension $\binom{n+d-1}{d}$.   
\end{remark} \noindent
A convex cone in $\RR^n$ is \emph{spectrahedral} if it is of the form
\begin{align*}
\displaystyle \left \{ \mathbf{x} \in \mathbb{R}^n : \sum_{i = 1}^n x_i A_i \thickspace \text{ is positive semidefinite} \right \}
\end{align*} \noindent
where $A_i$, $i = 1, \dots, n$ are symmetric matrices such that there exists a vector $(y_1, \dots, y_n) \in \mathbb{R}^n$ with $\sum_{i=1}^n y_i A_i$ positive definite. It is easy to see that spectrahedral cones are hyperbolicity cones. A major open question asks if the converse is true. 
\begin{conjecture}[Generalized Lax conjecture \cite{HV,Vin}] \label{glc}
All hyperbolicity cones are spectrahedral.
\end{conjecture} \noindent
We may reformulate Conjecture \ref{glc} as follows, see \cite{HV,Vin}. The hyperbolicity cone of $h(\mathbf{x})$ with respect to $\mathbf{e} = (e_1, \dots ,e_n)$ is spectrahedral if there is a homogeneous polynomial $q(\mathbf{x})$ and real symmetric matrices $A_1, \ldots, A_n$ of the same size such that
\begin{align}\label{qhd}
\displaystyle q(\mathbf{x})h(\mathbf{x}) = \det \left ( \sum_{i = 1}^n x_i A_i \right )
\end{align} \noindent
where $\Lambda_{++}(h, \mathbf{e}) \subseteq \Lambda_{++}(q, \mathbf{e})$ and $\sum_{i=1}^n e_iA_i$ is positive definite. 
\begin{itemize}
\item Conjecture \ref{glc} is true for $n=3$ by Theorem \ref{heltonvinnikov},
\item Conjecture \ref{glc} is true for homogeneous cones \cite{Chua}, i.e., cones for which the automorphism group acts transitively on its interior, 
\item Conjecture \ref{glc} is true for quadratic polynomials, see e.g. \cite{NT}, 
\item Conjecture \ref{glc} is true for elementary symmetric polynomials, see \cite{Bel},
\item Weaker versions of Conjecture \ref{glc} are true for smooth hyperbolic polynomials, see \cite{Kum2, NS}. 
\item Stronger algebraic versions of Conjecture \ref{glc} are false, see \cite{B}. 
\end{itemize}

A class of polynomials which is intimately connected to hyperbolic polynomials is the class of stable polynomials. Below we will collect a few facts about stable polynomials that will be needed in forthcoming sections. 
A polynomial $P(\xx) \in \mathbb{C}[x_1, \dots, x_n]$ is \emph{stable} if $P(z_1,\ldots,z_n) \neq 0$ whenever $\Im(z_j)>0$ for all $1\leq j \leq n$. Stable polynomials satisfy the following basic closure properties, see e.g. \cite{Wag}. 
\begin{lemma}\label{closureprop}
Let $P(x_1, \dots, x_n)$ be a stable polynomial of degree $d_i$ in $x_i$ for $i = 1, \dots, n$. Then for all $i = 1, \dots, n$ we have
\begin{enumerate}
\item Specialization: $P(x_1, \dots, x_{i-1}, \zeta, x_{i+1}, \dots, x_n)$ is stable or identically zero for each $\zeta \in \mathbb{C}$ with $\Im(\zeta) \geq 0$.
\item Scaling: $P(x_1, \dots, x_{i-1}, \lambda x_i, x_{i+1}, \dots, x_n)$ is stable for all $\lambda > 0$.
\item Inversion: $x_i^{d_i} P(x_1, \dots, x_{i-1}, - x_i^{-1}, x_{i+1}, \dots, x_n)$ is stable.
\item Permutation: $P(x_{\sigma(1)}, \dots, x_{\sigma(n)})$ is stable for all $\sigma \in \mathfrak{S}_n$.
\item Differentiation: $\partial/ \partial x_i P(x_1, \dots, x_n)$ is stable.
\end{enumerate}
\end{lemma} \noindent
Hyperbolic and stable polynomials are related as follows, see \cite[Prop. 1.1]{PLMS} and \cite[Thm. 6.1]{COSW}. 
\begin{lemma}
\label{hypbas}
Let $P \in \mathbb{R}[x_1, \dots, x_n]$ be a homogenous polynomial. Then $P$ is stable if and only if $P$ is hyperbolic with $\mathbb{R}_+^n \subseteq \Lambda_+$. 

Moreover all non-zero Taylor coefficients of a homogeneous and stable polynomial have the same phase, i.e., the quotient of any two non-zero coefficients is a positive real number.
\end{lemma} \noindent
\begin{lemma}[Lemma 4.3 in \cite{B}]
\label{hypstab}
If $h \in \mathbb{R}[y_1, \dots, y_n]$ is a hyperbolic polynomial, $\vv_1, \dots, \vv_m \in \Lambda_+$ and $\vv_0 \in \mathbb{R}^n$, then the polynomial
\begin{align*}
\displaystyle P(\xx) = h(\vv_0 + x_1\vv_1 + \cdots + x_m \vv_m)
\end{align*} \noindent
is either identically zero or stable.
\end{lemma}

\section{Hyperbolic polymatroids}
We refer to \cite{O} for undefined matroid terminology.
The connection between hyperbolic/stable polynomials and matroids was first realized in \cite{COSW}. A polynomial is \emph{multiaffine} provided that each variable occurs at most to the first power. Choe \emph{et al.} \cite{COSW} proved that if 
\begin{equation}\label{Pbases}
P(\xx)= \sum_{B \subseteq [m]} a(B) \prod_{i \in B}x_i \in \CC[x_1,\ldots, x_m]
\end{equation}
is a homogeneous, multiaffine and stable polynomial, then its \emph{support} $$\BB=\{B : a(B) \neq 0\}$$ is the set of bases of a matroid, $\MM$, on $[m]$. Such matroids are called \emph{weak half--plane property matroids} (abbreviated WHPP--matroids). If further $P(\xx)$ can be chosen so that $a(B) \in \{0,1\}$, then $\MM$ is called a \emph{half--plane property matroid} (abbreviated HPP--matroid). If so, then $P(\xx)$ is the \emph{bases generating polynomial} of $\MM$. 
\begin{itemize}
\item All matroids representable over $\CC$ are WHPP,  \cite{COSW}. 
\item A binary matroid is WHPP if and only if it is HPP if and only if it is   regular, \cite{BG, COSW}. 
\item No finite projective geometry $\rm{PG}(r,n)$ is WHPP, \cite{BG, COSW}. 
\item The V\'amos matroid $V_8$ is HPP (but not representable over any field), \cite{WW}. 
\end{itemize} \noindent

We shall now see how weak half-plane property matroids may conveniently be described in terms of hyperbolic polynomials. 

Let $E$ be a finite set. A \emph{polymatroid} is a function $r : 2^E \rightarrow \NN$ satisfying 
\begin{enumerate}
\item $r(\emptyset)=0$,
\item $r(S) \leq r(T)$ whenever $S \subseteq T \subseteq E$, 
\item $r$ is \emph{semimodular}, i.e., 
$$r(S)+r(T) \geq r(S\cap T) +r(S\cup T),$$ 
for all $S,T \subseteq E$. 
\end{enumerate}
Recall that rank functions of matroids on $E$ coincide  polymatroids $r$ on $E$ with $r(\{i\}) \leq 1$ for all $i \in E$. 

 Let $\VV = (\vv_1, \ldots, \vv_m)$ be a tuple of vectors in 
$\Lambda_+(h,\ee)$, where $\ee \in \RR^n$. The \emph{(hyperbolic) rank}, $\rk(\xx)$, of $\xx \in \RR^n$ is defined to be the number of non-zero eigenvalues of $\xx$, i.e., 
$
\rk(\xx)= \deg h(\ee +t\xx).
$
Define a function $r_\VV : 2^{[m]} \rightarrow \NN$, where $[m]:=\{1,2,\ldots, m\}$, by 
$$
r_\VV(S) = \rk\left(\sum_{i \in S}\vv_i\right). 
$$
It follows from \cite{Gur} (see also \cite{B}) that $r_\VV$ is a polymatroid. We call such polymatroids \emph{hyperbolic polymatroids}. Hence if the vectors in $\VV$ have rank at most one, then we obtain the hyperbolic  rank function of a \emph{hyperbolic matroid}. 
\begin{example}
Let $A_1=\uu_1 \uu_1^*, \ldots, A_m=\uu_m \uu_m^*$ be PSD matrices of rank at most one in $\CC^n$. By Example \ref{detta} the function 
$r : 2^{[m]} \rightarrow \NN$ defined by 
$$
r(S) = \rk\left(\sum_{i \in S}A_i\right)
$$
is the rank function of a hyperbolic matroid. It is not hard to see that $r(S)$ is equal to the dimension of the subspace of $\CC^n$ spanned by $\{\uu_i : i \in S\}$. Hence $r$ is the rank function of the linear matroid defined by $\uu_1,\ldots, \uu_m$. 
\end{example}
\begin{proposition}
\label{whpphyp}
A matroid is hyperbolic if and only it has the weak half--plane property. 
\end{proposition}
\begin{proof}
Suppose $\BB$ is the set of bases of a matroid, $\MM $, with the weak half--plane property realized by \eqref{Pbases}. By Lemma \ref{hypbas} we may assume that $a(B)$ is a  nonnegative real number for all $B \subseteq [m]$. Then $P(\xx)$ is hyperbolic with hyperbolicity cone containing the positive orthant by Lemma \ref{hypbas}. Let $\VV=(\delta_1, \ldots, \delta_m)$, be the standard basis of $\RR^m$, and let $\one =(1,\ldots,1) \in \RR^m$ be the all ones vector. Then 
\begin{align*}
r_\VV(S) &= \rk\left(\sum_{i \in S}\delta_i\right) 
= \deg P\left(\one+t\sum_{i \in S}\delta_i\right) \\
&= \deg \sum_{B}a(B)(1+t)^{|B\cap S|} = \max\{ |B \cap S| : B \in \BB\}, 
\end{align*}
and hence $r_\VV$ is the rank function of $\MM$. 

Assume $h$ is hyperbolic and $r_\VV$, where $\VV=(\vv_1,\ldots, \vv_m) \in \Lambda_+(h,\ee)^m$, is the rank function of a hyperbolic matroid of rank $r$. We may assume $h(\ee)>0$. The polynomial $g(x_0,x_1, \ldots, x_m)= h(x_0\ee+x_1\vv_1+\cdots+x_m\vv_m)$ is stable by Lemma \ref{hypstab} and has nonnegative coefficients only by Lemma \ref{hypbas}. Since $\vv_i$ has rank at most one for each $i$ we see that $g$ has degree at most one in $x_i$ for all $i \geq 1$. It follows that 
$$
g(\xx)= x_0^{d-r} \sum_{i=0}^r g_i(x_1,\ldots, x_m)x_0^{r-i},
$$
where $g_i(\xx)$ is a homogeneous and multiaffine polynomial of degree $i$ for $0 \leq i \leq r\leq d=\deg h$. By dividing by $x_0^{d-r}$ and setting $x_0=0$, we see that $g_r(\xx)$ is a stable by Lemma \ref{closureprop}. Moreover $B$ is a basis of the matroid defined by $\VV$ if and only if $|B|=r$ and $g(\delta_0 + t\sum_{i \in B}\delta_i)$ has degree $d$. This happens if and only if $g_r(\sum_{i \in B}\delta_i) \neq 0$, that is, if and only if $B$ is in the support of $g_r(\xx)$. 
\end{proof} \noindent
\section{Projections and face lattices of hyperbolicity cones}
Let $C$ be a closed convex cone in $\RR^n$. If $\xx, \yy \in C$ and $\yy-\xx \in C$, we write $\xx \leq \yy$. 
Recall that a \emph{face} $F$ of a convex cone $C$ is a convex subcone of $C$ with the property that 
$\xx, \yy \in C$, $\xx \leq \yy$ and $\yy \in F$ implies $\xx \in F$. Equivalently a face is a convex subcone of $C$ such that for each open line segment in $C$ that intersects $F$, the closure of the segment is contained in $F$. The collection of all faces of $C$ is a lattice, $L(C)$, under containment with smallest element $\{0\}$ and largest element $C$. Clearly $F\wedge G= F\cap G$ and $F \vee G = \bigcap_{H}H$, where $H$ ranges over all faces containing $F$ and $G$. The collection of all relative interiors of faces of $C$ partitions $C$. If $F_\xx$ is the unique face that contains $\xx \in C$ in its relative interior, then $F_\xx \vee F_\yy = F_{\xx+\yy}$. See \cite{Rock} for more on the face lattices of convex cones. 

The \emph{rank} of a face $F$ of the hyperbolicity cone $\Lambda_+$ is defined by 
$$
\rk(F)= \max_{\xx \in F} \rk(\xx).
$$ \noindent
Note that if $L(\Lambda_+)$ is a graded lattice, then the above hyperbolic rank function is not necessarily the rank function of $L(\Lambda_+)$. 
\begin{lemma}[Thm 26, \cite{R}] \label{rkint}
Let $F$ be a face of $\Lambda_+$ and let $\xx \in F$. Then 
$\rk(\xx) = \rk(F)$  if and only if $\xx$ is in the relative interior of $F$.  
\end{lemma} \noindent

By Lemma \ref{rkint} and the semimodularity of hyperbolic polymatroids we see that $\rk: L(\Lambda_+) \to \mathbb{N}$ is \emph{semimodular}, that is, 
\begin{align*}
\rk(F \vee G) + \rk(F \wedge G) \leq \rk(F) + \rk(G)
\end{align*} \noindent
for all $F,G \in L(\Lambda_+)$. We may therefore equivalently define a hyperbolic polymatroid in terms of the face lattice of the hyperbolicity cone as follows: If $\FF =(F_1, \ldots, F_m)$ is a tuple of elements of  the face lattice $L(\Lambda_+)$, then the function $r_\FF : 2^{[m]} \rightarrow \NN$ defined by 
$$
r_\FF(S) = \rk\left(\bigvee_{i \in S}F_i \right)
$$
is a hyperbolic polymatroid.

The following theorem collects a few fundamental facts about hyperbolic polynomials and their hyperbolicity cones. For proofs see \cite{Gar,R}. 

\begin{theorem}[G\aa rding, \cite{Gar}]\label{hypfund}
Suppose $h$ is hyperbolic with respect to $\ee \in \RR^n$. 
\begin{enumerate}
\item $\Lambda_+(\ee)$ and $\Lambda_{++}(\ee)$ are convex cones.
\item $\Lambda_{++}(\ee)$ is the connected component of 
$$
\{ \xx \in \RR^n : h(\xx) \neq 0\}
$$
which contains $\ee$. 
\item $\lmi : \RR^n \rightarrow \RR$ is a concave function, and  $\lma : \RR^n \rightarrow \RR$ is a convex function. 
\item If $\ee' \in \Lambda_{++}(\ee)$, then $h$ is hyperbolic with respect to $\ee'$ and $\Lambda_{++}(\ee')=\Lambda_{++}(\ee)$.
\end{enumerate}
\end{theorem} 
\noindent
Recall that the \emph{lineality space} of a convex cone $C$ is $C \cap -C$, i.e., the largest linear space contained in $C$. It follows that the lineality space of a hyperbolicity cone is $\{\xx : \lambda_i(\xx)=0 \mbox{ for all } i\}$, see e.g. \cite{R}. Also if $\xx$ is in the lineality space, then $\lambda_i(\xx+\yy) = \lambda_i(\yy)$ for all $1\leq i \leq d$ and $\yy \in \RR^n$ \cite{R}. 

By homogeneity of $h$ 
\begin{equation}\label{dilambdas}
\lambda_j(s\xx+t\ee)= 
\begin{cases}
s\lambda_j(\xx)+t &\mbox{ if } s\geq 0 \mbox{ and } \\
s\lambda_{d-j+1}(\xx)+t &\mbox{ if } s \leq 0
\end{cases},
\end{equation}
for all $s,t \in \RR$ and $\xx \in \RR^n$.

In analogy with the eigenvalue characterization of matrix projections we define projections in $\Lambda_+$ as follows.
\begin{definition}
An element in $\Lambda_+$ is a \emph{projection} if its eigenvalues are contained in $\{0,1\}$.
\end{definition}
\begin{remark}
Note that $\mathbf{0}$, $\ee$ and appropriate multiples of rank one vectors in $\Lambda_+$ are always projections.
\end{remark}
\begin{lemma}\label{faceproj}
Suppose $\xx, \yy \in \Lambda_+$ are such that $F_\xx \leq F_\yy$ and $\rk(\yy) =r$. 
If $\lambda_1(\xx) \leq \lambda_r(\yy)$, then $\xx \leq \yy$. 

In particular if $\xx, \yy \in \Lambda_+$ are projections, then $F_\xx \leq F_\yy$ if and only if $\xx \leq \yy$. 
\end{lemma}
\begin{proof}
Suppose $\xx, \yy \in \Lambda_+$ are such that $F_\xx \leq F_\yy$, $\rk(\yy) =r$ and 
 $\lambda_1(\xx) \leq \lambda_r(\yy)$. 
Consider the polynomial
$$
g(u,s,t)= h(u\ee+s\yy+t\xx),
$$
which is hyperbolic with respect to $(1,0,0)$ and whose hyperbolicity cone contains the positive orthant. Since $\xx \in F_\yy$ we know that $\rk(a\xx+b\yy) =r$ for all 
$a,b >0$. Since all non-zero Taylor coefficients of $g(u,s,t)$ have the same sign, by Lemma \ref{hypbas}, we may write 
$$
g(u,s,t)= u^{d-r}g_0(u,s,t), \quad d=\deg h, 
$$
where $g_0(u,s,t)$ is hyperbolic with respect to $(1,0,0)$ and also $(0,1,0)$, and its hyperbolicity cone contains the positive orthant.  Let $\lambda'_j(a,b,c)$, $j=1,\ldots, r$, denote the eigenvalues of $g_0$ (with respect to $(1,0,0)$). Then by \eqref{dilambdas} and the concavity of $\lambda'_r$ (Theorem \ref{hypfund}):
$$
\lambda'_r(0,1,-1) \geq \lambda'_r(0,1,0) + \lambda'_r(0,0,-1)=\lambda_r(\yy)-\lambda_1(\xx) \geq 0. 
$$
By construction $\lmi(\yy-\xx)= \min\{0, \lambda'_r(0,1,-1)\}$, and  the  lemma follows. 
\end{proof}
\begin{lemma}\label{notlaxred}
If $\xx,\yy \in \Lambda_+$ are projections with $\xx \leq \yy$, then $\yy-\xx$ is a projection with 
$$
\rk(\yy-\xx) = \rk(\yy)-\rk(\xx).
$$
\end{lemma}
\begin{proof}
Suppose first $F_\yy = \Lambda_+ = F_\ee$. Then $\yy-\ee, \ee-\yy \in \Lambda_+$ by Lemma \ref{faceproj}, and hence $\yy-\ee$ is in the lineality space of $\Lambda_+$. Then 
$$
\lambda_{i}(\yy-\xx)= \lambda_i(\ee-\xx)= 1-\lambda_{d-i+1}(\xx),
$$
for all $1 \leq i \leq d=\deg h$, and hence $\yy-\xx$ is a projection of rank $d-\rk(\xx)$. 

If $F_\yy \neq F_\ee$, then $r:=\rk(\yy)<d$. Consider the hyperbolic polynomial 
$$
g(u,s,t)= h(u\ee+s\xx+t\yy)=u^{d-r}g_0(u,s,t),
$$
where $g_0$ is hyperbolic with respect to $\ee'=(1,0,0)$. It follows that $\xx'=(0,1,0)$ and 
$\yy'=(0,0,1)$ are projections with $F_{\ee'}=F_{\yy'}$. The lemma now follows from the first case considered. 
\end{proof} \noindent
\begin{remark}
Note that if $\xx \leq \yy$ and $\yy \leq \xx$, then $\yy-\xx$ is in the lineality space of $\Lambda_+$. Moreover $\xx \leq \yy$ if and only if $\ee-\yy \leq \ee -\xx$. Since $F_{\ee} = \Lambda_+$ we have by Lemma \ref{faceproj} that $\xx \leq \ee$ for all projections $\xx \in \Lambda_+$. Hence by Lemma \ref{notlaxred} it follows that $\xx$ is a projection if and only if $\ee - \xx$ is a projection.
\end{remark} \noindent
The following proposition gives a sufficient condition for two faces in $\Lambda_+$ to be modular with respect to the hyperbolic  rank function. 
\begin{proposition}
\label{modfacelat}
If $\xx, \yy \in \Lambda_+$ are projections such that $F_\xx \wedge F_\yy$, 
$F_\xx \vee F_\yy$, $F_{\ee-\xx} \wedge F_{\ee-\yy}$ and 
$F_{\ee-\xx} \vee F_{\ee-\yy}$ all contain a projection in their relative interiors, then 
$$
\rk(F_\xx)+\rk(F_\yy) = \rk(F_\xx \wedge F_\yy)+ \rk(F_\xx \vee F_\yy).
$$
\end{proposition}
\begin{proof}
Let $\vv, \ww, \vv', \ww'$ be the projections in the relative interiors of $F_\xx \wedge F_\yy$, 
$F_\xx \vee F_\yy$, $F_{\ee-\xx} \wedge F_{\ee-\yy}$ and $F_{\ee-\xx} \vee F_{\ee-\yy}$, respectively. Then $\ee-\ww \leq \ee-\xx$ and $\ee-\ww \leq \ee-\yy$, so that 
$\ee-\ww \in F_{\ee-\xx} \wedge F_{\ee-\yy}$ by Lemma \ref{faceproj}. By Lemma \ref{faceproj} again, $\ee-\ww \leq \vv'$. We also have $\ee-\vv' \geq \xx$ and $\ee-\vv' \geq \yy$ so that  $\ee-\vv' \geq \ww$, that is, $\ee-\ww \geq \vv'$. Thus $F_{\vv'}=F_{\ee-\ww}$ and analogously $F_{\ww'}=F_{\ee-\vv}$. Since $\rk : L(\Lambda_+) \rightarrow \NN$ is semimodular we have 
$$
\rk(\xx)+ \rk(\yy) \geq \rk(\vv) + \rk(\ww),
$$
and also 
$$
\rk(\ee-\xx)+ \rk(\ee-\yy) \geq \rk(\ee-\ww) + \rk(\ee-\vv),
$$ 
and so the proposition follows from Lemma \ref{notlaxred}. 
\end{proof} \noindent
\begin{corollary}\label{modgeom}
Let $\Lambda_+(h,\ee)$ be a hyperbolicity cone with trivial lineality space. Suppose all extreme rays of $\Lambda_+$ have the same hyperbolic rank, and that each face of $\Lambda_+$ contains a projection in its relative interior. Then $L(\Lambda_+)$ is a modular geometric lattice. 
\end{corollary}
\begin{proof}
Since each face of $L(\Lambda_+)$ except $\{0\}$ is generated by extreme rays, see e.g. \cite[Cor. 18.5.2]{Rock}, it follows that  $L(\Lambda_+)$ atomic. Suppose $\rk(\mathbf{a})=c$ for all atoms $\mathbf{a} \in L(\Lambda_+)$. By modularity of the hyperbolic rank function (Proposition \ref{modfacelat}) and induction we see that $c$ divides 
$\rk(F)$ for all $F \in L(\Lambda_+)$. It follows that the function defined by 
$\rk(F)/c$ is the proper rank function of $L(\Lambda_+)$, since it is modular and equal to one on each atom. 
\end{proof}
\section{Hyperbolic matroids and Euclidean Jordan algebras} \noindent
\begin{figure}
\includegraphics[width=100mm]{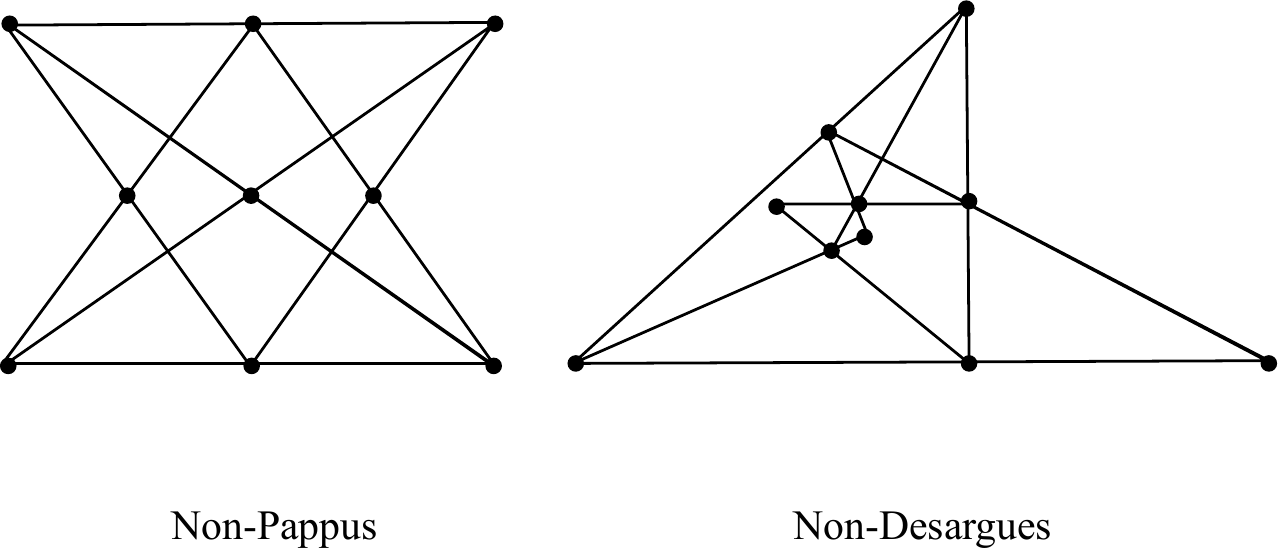}
\caption{The Non-Pappus and Non-Desargues configurations.}
\label{fanopappusdesargues}
\end{figure} 
\noindent In light of the generalized Lax conjecture it is of interest to find hyperbolic but non-linear (poly-) matroids.  Until present the only known instances of non-linear hyperbolic matroids are the V\'amos matroid \cite{WW} and a generalization of it \cite{BVY}. The generalized V\'amos matroids introduced in the following section provide an infinite family of such matroids. In this section we identify two further types of matroids that are hyperbolic but not linear through a connection with Euclidean Jordan algebras and projective geometry. \newline 
\newline Some classical examples of non-linear matroids are obtained by relaxing a circuit hyperplane in a matroid that comes from a geometric configuration. In fact the Non-Fano, Non-Pappus and Non-Desargues matroids (see Fig \ref{fanopappusdesargues}) are all derived from the family $n_3$ of symmetric configurations on $n$ points and $n$ lines, arranged such that $3$ lines pass through each point and $3$ points lie on each line \cite{Gru}. Note that such configurations need not be unique up to incidence isomorphism for given $n$. The Non-Fano, Non-Pappus and Non-Desargues matroids are all rank three matroids corresponding respectively to instances of the configurations $7_3, 9_3$ and $10_3$ after removing one line. It is interesting to note how representability diminishes as we move upwards in this hierarchy: The Non-Fano matroid is representable over all fields that do not have characteristic $2$ \cite{O}. The Non-Pappus matroid is skew-linear but not linear \cite{Ing}, which is to say that it only admits representations over non-commutative division rings e.g. the quaternions $\mathbb{H}$. Moreover it is known that the Non-Desargues matroid is not even skew-linear \cite{Ing}. On the other hand, it is known that the Non-Desargues matroid can be coordinatized by rank one projections over the octonions $\mathbb{O}$, see e.g. \cite{GPR}. The octonions form a non-commutative and non-associative division ring over the reals.
\newline 
\newline
An algebra $(A, \circ)$ over a field $\mathbb{K}$ is said to be a \emph{Jordan algebra} if for all $a,b \in A$
\begin{align*}
\displaystyle a \circ b = b \circ a \hspace{0.5cm} \text{ and } \hspace{0.5cm} a \circ (a^2 \circ b) = a^2 \circ (a \circ b).
\end{align*} \noindent
A Jordan algebra is \emph{Euclidean} if 
\begin{align*}
\displaystyle a_1^2 + \cdots + a_k^2 = 0 \implies a_1 = \cdots = a_k = 0
\end{align*} \noindent
for all $a_1, \dots, a_k \in A$. 
By a theorem of Jordan, von Neumann and Wigner \cite{JNW} the simple finite dimensional real Euclidean Jordan algebras classify into four infinite families and one exceptional algebra (the Albert algebra) as follows:
\begin{enumerate}
\item $H_n(\mathbb{K})$ ($\mathbb{K} = \mathbb{R}, \mathbb{C}, \mathbb{H}$) - the algebra of  Hermitian $n \times n$ matrices over $\mathbb{K}$ with Jordan product $a \circ b = \frac{1}{2}(ab + ba)$.
\item $\mathbb{R}^n \oplus \mathbb{R}$ - the real inner product space with inner product $(u \oplus \lambda, v \oplus \mu) = (u,v)_{\mathbb{R}^n} + \lambda \mu$ and Jordan product $(u \oplus \lambda) \circ (v \oplus \mu) = (\mu u+ \lambda v) \oplus ((u,v)_{\mathbb{R}^n} + \lambda \mu)$.
\item $H_3(\mathbb{O})$ - the algebra of octonionic Hermitian $3 \times 3$ matrices with Jordan product $a \circ b = \frac{1}{2}(ab + ba)$. 
\end{enumerate} \noindent
We refer to \cite{FK} for facts about Euclidean Jordan algebras mentioned below.
Let $A$ be a real Euclidean Jordan algebra of rank $r$ with identity $e$. A \emph{Jordan frame} is a complete system of orthogonal idempotents of rank one, that is, rank one elements $c_1, \dots, c_r \in A$ such that $c_i^2 = c_i$, $c_i \circ c_j = 0$ for $i \neq j$ and $c_1 + \cdots + c_r = e$.
A characteristic property of finite dimensional real Euclidean Jordan algebras is the spectral theorem 
\begin{theorem}
\label{spectral}
Let $A$ be a real Euclidean Jordan algebra of rank $r$.
Then for each $x \in A$ there exists a Jordan frame $c_1, \dots, c_r \in A$ and unique real numbers $\lambda_1(x), \dots, \lambda_r(x)$ (the eigenvalues) such that 
\begin{align*}
\displaystyle x = \sum_{j=1}^r \lambda_j(x) c_j.
\end{align*} \noindent
Moreover 
\begin{align*}
\displaystyle \sum_{j : \lambda_j = \lambda} c_j
\end{align*} \noindent
is uniquely determined for each eigenvalue $\lambda$.
\end{theorem} \noindent
A finite dimensional real Euclidean Jordan algebra is equipped with a hyperbolic determinant polynomial $\det:A \to \RR$ given by
\begin{align*}
\displaystyle \det(x) = \prod_{j=1}^r \lambda_j(x).
\end{align*} \noindent
Let $P$ be a set of \emph{points} and $L$ a set of \emph{lines}. Recall that a pair $G = (P,L)$ is a \emph{projective geometry} if the following axioms are satisfied:
\begin{enumerate}
\item For any two distinct points $a,b \in P$ there is a unique line $ab \in L$ containing $a$ and $b$.
\item Any line contains at least three points.
\item If $a,b,c,d \in P$ are distinct points such that $ab \cap cd \neq \emptyset$ then $ac \cap bd \neq \emptyset$.
\end{enumerate} 
Each projective geometry is a (simple) modular geometric lattice, and each modular geometric lattice is a direct product of a Boolean algebra with projective geometries, see 
\cite[p.~93]{Bir}. The following proposition is essentially a known connection between Jordan algebras and projective geometries, which we here prove in the theory of hyperbolic polynomials.  
\noindent
\begin{proposition}
\label{facelatprojgeom}
Let $A$ be a finite dimensional real Euclidean Jordan algebra and let $\Lambda_+$ denote the hyperbolicity cone of $\det:A \to \RR$. Then $L(\Lambda_+)$ is a modular geometric lattice. 

In particular if $A$ is simple, then $L(\Lambda_+)$ is a projective geometry. 
\end{proposition}
\begin{proof}
By Theorem \ref{spectral} the extreme rays of $\Lambda_+$ are multiples of rank one idempotents. Also, a face $F_x$ contains the projection 
$$
c= \sum_{j : \lambda_j(x)\neq 0 } c_j
$$
in its relative interior. The proposition now follows from Corollary \ref{modgeom}. 
\end{proof} \noindent
The Non-Pappus and Non-Desargues configurations are depicted in Fig \ref{fanopappusdesargues}. 
The configurations give rise to rank $3$ matroids where three points are dependent if and only if they are collinear. The Non-Pappus and Non-Desargues matroids are not linear but may be  represented over the projective geometries associated to the Euclidean Jordan algebras $H_3(\mathbb{H})$ and $H_3(\mathbb{O})$, respectively. This may be deduced from the coordinatizations in \cite[Example 1.5.14]{O} and \cite{GPR}. Hence by Proposition \ref{facelatprojgeom} we have
\begin{theorem}
The Non-Pappus and Non-Desargues matroids are hyperbolic.
\end{theorem}
\section{Generalized V\'amos Matroids with the (weak) half--plane property}
In this section we provide an infinite family of hyperbolic matroids that do not arise from modular geometric lattices. Let us be precise. Suppose $L$ is a lattice with a smallest element $\hat{0}$, and 
$f : L \rightarrow \NN$ is a function satisfying 
\begin{enumerate}
\item $f(\hat{0})=0$,
\item if $x \leq y$, then $f(x) \leq f(y)$, 
\item for any $x,y \in L$, 
$$
f(x)+f(y) \geq f(x\vee y) + f(x \wedge y).
$$
\end{enumerate}
If $x_1,\ldots, x_m \in L$, then the function $r : 2^{[m]} \rightarrow \NN$ defined by 
$$
r(S) = f\left(\bigvee_{i \in S} x_i\right)
$$
defines a polymatroid. All polymatroids arise in this manner. However 
if $f$ is modular, i.e., 
$$
f(x)+f(y) = f(x\vee y) + f(x \wedge y), \ \ \ \mbox{ for all } x,y \in L,
$$
we say that $r$ is \emph{modularly represented}. Hence all linear matroids as well as all projective geometries are modularly represented. Although Ingleton's proof \cite{Ing} of the next lemma only concerns linear matroids it extends verbatim to modularly represented matroids. 
\begin{lemma}[Ingleton's Inequality, \cite{Ing}]\label{Ingin}
Suppose $r : 2^{E} \rightarrow \NN$ is a modularly represented polymatroid and 
$A,B,C,D \subseteq E$. Then 
\begin{align*}
\displaystyle & r(A \cup B) + r(A \cup C \cup D) + r(C) + r(D) + r(B \cup C \cup D) \leq \\ & r(A \cup C) + r(A \cup D) + r(B \cup C) + r(B \cup D) + r(C \cup D).
\end{align*} \noindent
\end{lemma}

\noindent The \emph{V\'amos matroid} $V_8$ is the rank-four matroid on $E = [8]$ having set of bases 
\begin{align*}
\mathcal{B}(V_8) = \binom{E}{4} \setminus \{ \{1,2,3,4\}, \{1,2,5,6 \}, \{1,2,7,8 \}, \{3,4,5,6 \}, \{5,6,7,8 \} \}.
\end{align*} \noindent 

The rank function of the V\'amos matroid fails to satisfy Ingleton's inequality (see \cite{Ing}), and hence it is not modularly represented. Nevertheless Wagner and Wei \cite{WW} proved that $V_8$ has the half-plane property, and hence $V_8$ is hyperbolic. This was used in \cite{B} to provide counterexamples to stronger algebraic versions of the generalized Lax conjecture. 

Burton, Vinzant and Youm \cite{BVY} studied an infinite family of generalized V\'amos matroids, $\{V_{2n}\}_{n\geq 4}$, and conjectured that all members of the family have the half-plane property. They proved their conjecture for $n=5$. Below we generalize their construction and construct a family of matroids; one matroid for each uniform hypergraph. We prove that all matroids corresponding to simple graphs are HPP, and that all matroids corresponding to uniform hypergraphs are WHPP. In particular this will prove the conjecture of Burton \emph{et al.}

Recall that a rank $r$ \emph{paving matroid} is  matroid such that all its circuits have size at least $r$. Paving matroids may be characterized in terms of $d$-{partition}. A $d$-\emph{partition} of a set $E$ is a collection $\mathcal{S}$ of subsets of $E$ all of size at least $d$, such that every $d$-subset of $E$ lies in a unique member of $\mathcal{S}$. The $d$-partition $\mathcal{S}=\{E\}$ is the \emph{trivial} $d$-\emph{partition}. For a proof of the next proposition see \cite[Prop. 2.1.21]{O}. 

\begin{proposition}\label{d-part}
The hyperplanes of any rank $d+1 \geq 2$  paving matroid form a non-trivial $d$-partition. 

Conversely, the elements of a non-trivial $d$-partition form the set of hyperplanes of a paving matroid of rank $d+1$. 
\end{proposition}
A paving matroid of rank $r$ is \emph{sparse} if its hyperplanes all have size $r-1$ or $r$. 

Recall that a \emph{hypergraph} $H$ consists of a set $V(H)$ of \emph{vertices} together with a set $E(H) \subseteq 2^{V(H)}$ of \emph{hyperedges}. We say that a hypergraph $H$ is \emph{$d$-uniform} if all hyperedges have size $d$. 

\begin{theorem}
Let $H$ be an $d$-uniform hypergraph on $[n]$, and  let $E = \{ 1, 1', \dots,  n,n' \}$. Then 
\begin{align*}
\displaystyle \mathcal{B}(V_{H}) = \binom{E}{2d} \setminus \{ e\cup e' : e \in E(H) \},
\end{align*} \noindent
where $e':= \{i' : i \in e\}$, is the set of bases of a sparse paving matroid $V_H$ of rank $2d$. 
\end{theorem} \noindent
\begin{figure}
\makebox[\linewidth][c]{%
\begin{subfigure}[b]{.5\textwidth}
\includegraphics[width=60mm]{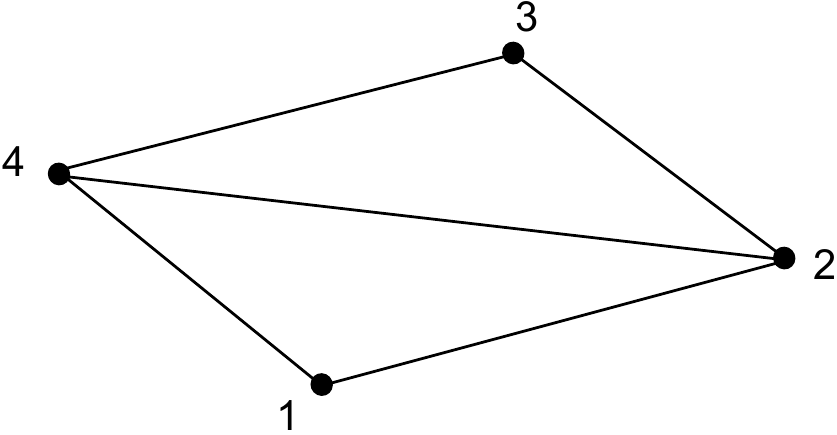}
\caption{$G$ (Diamond graph)}
\label{diamond}
\end{subfigure}%
\begin{subfigure}[b]{.5\textwidth}
\includegraphics[width=70mm]{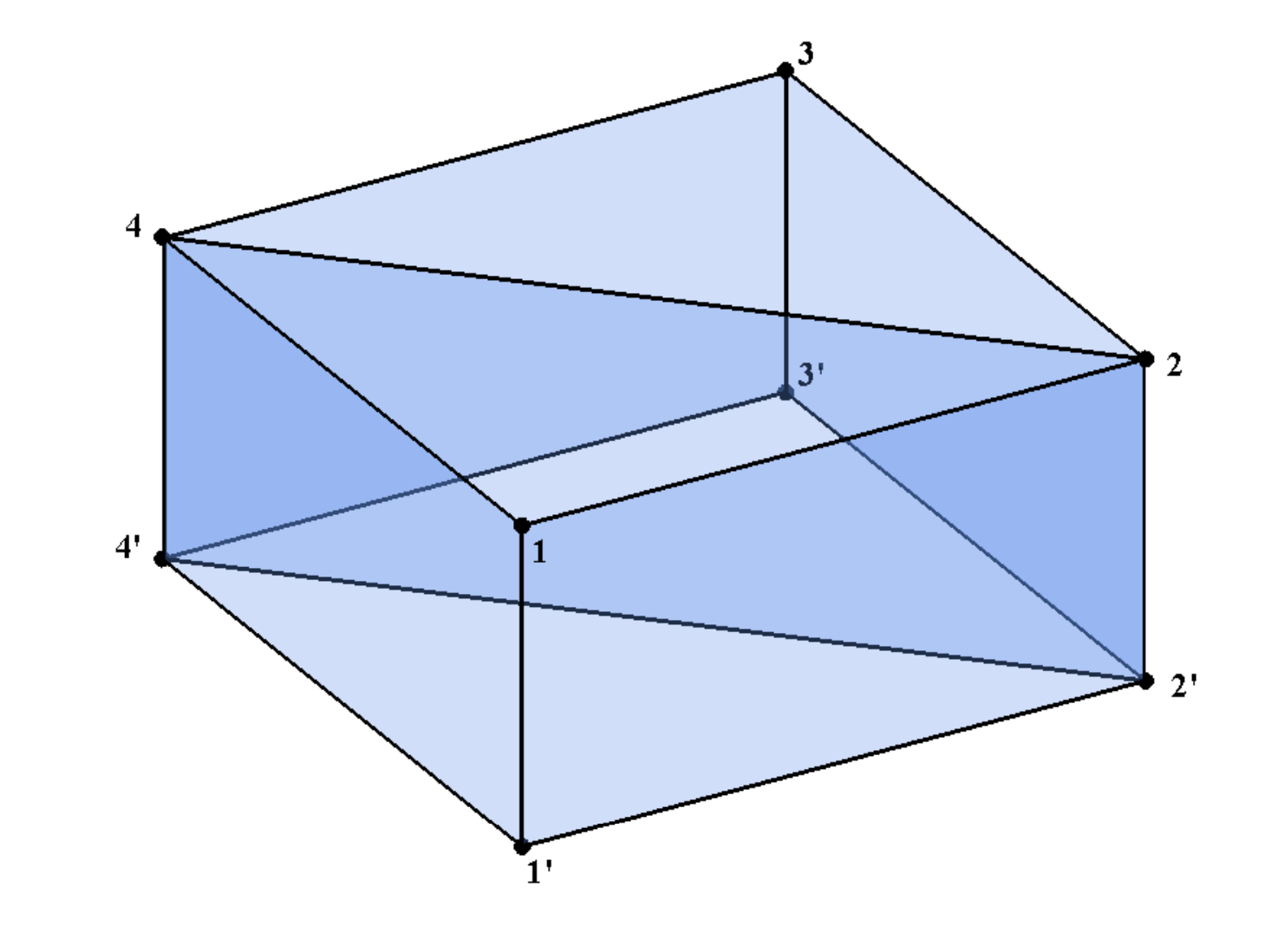}
\caption{$V_G \cong V_8$ (V\'amos matroid)}
\end{subfigure}
}
\end{figure}
\begin{figure}
\makebox[\linewidth][c]{%
\begin{subfigure}[b]{.5\textwidth}
\includegraphics[width=50mm]{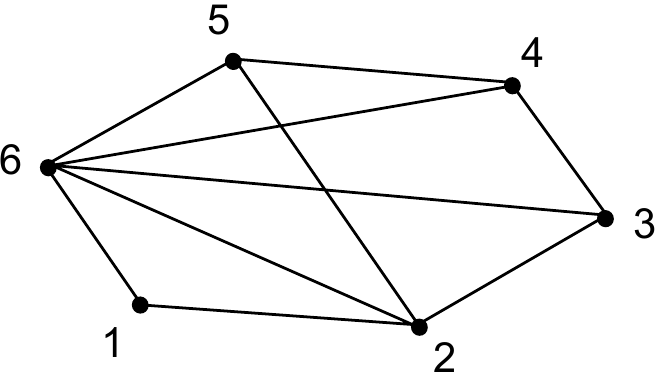}
\caption{A simple graph $G$}
\end{subfigure}%
\begin{subfigure}[b]{.5\textwidth}
\includegraphics[width=70mm]{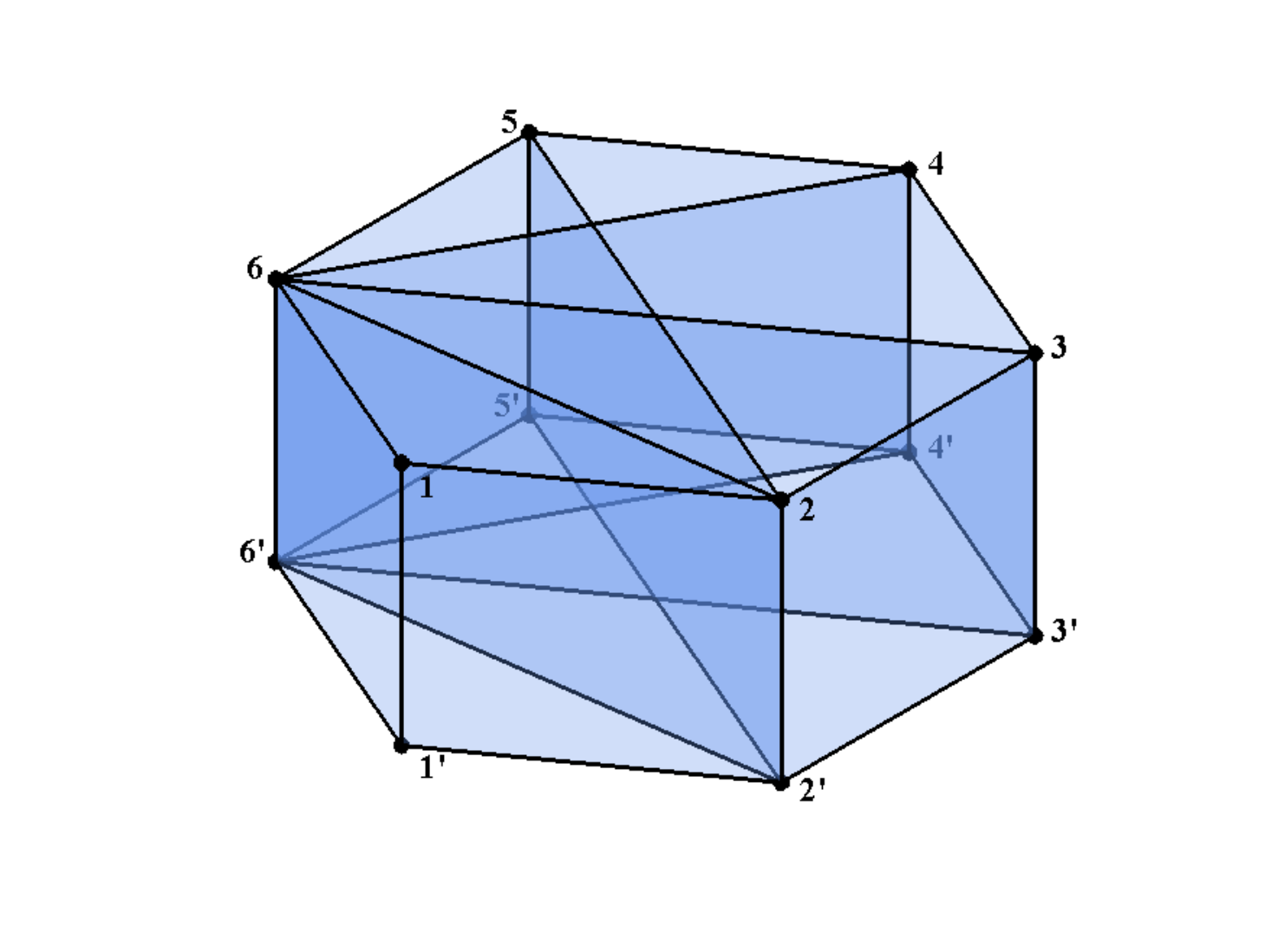}
\caption{The matroid $V_G$}
\end{subfigure}
}
\end{figure}
\begin{figure}
\makebox[\linewidth][c]{%
\begin{subfigure}[b]{.5\textwidth}
\includegraphics[width=50mm]{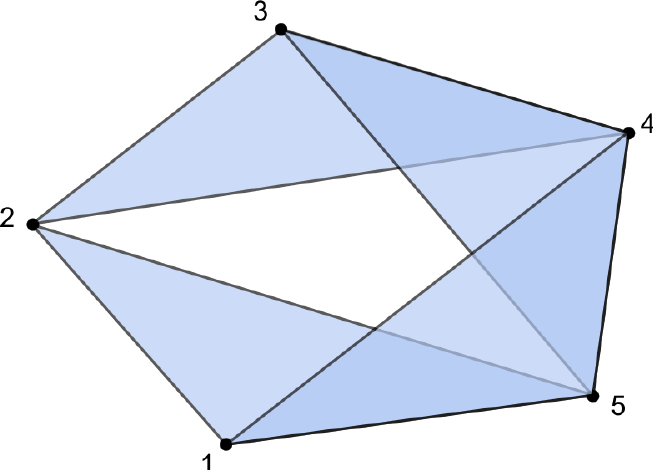}
\caption{A $3$-uniform hypergraph $H$}
\end{subfigure}%
\begin{subfigure}[b]{.5\textwidth}
\includegraphics[width=80mm]{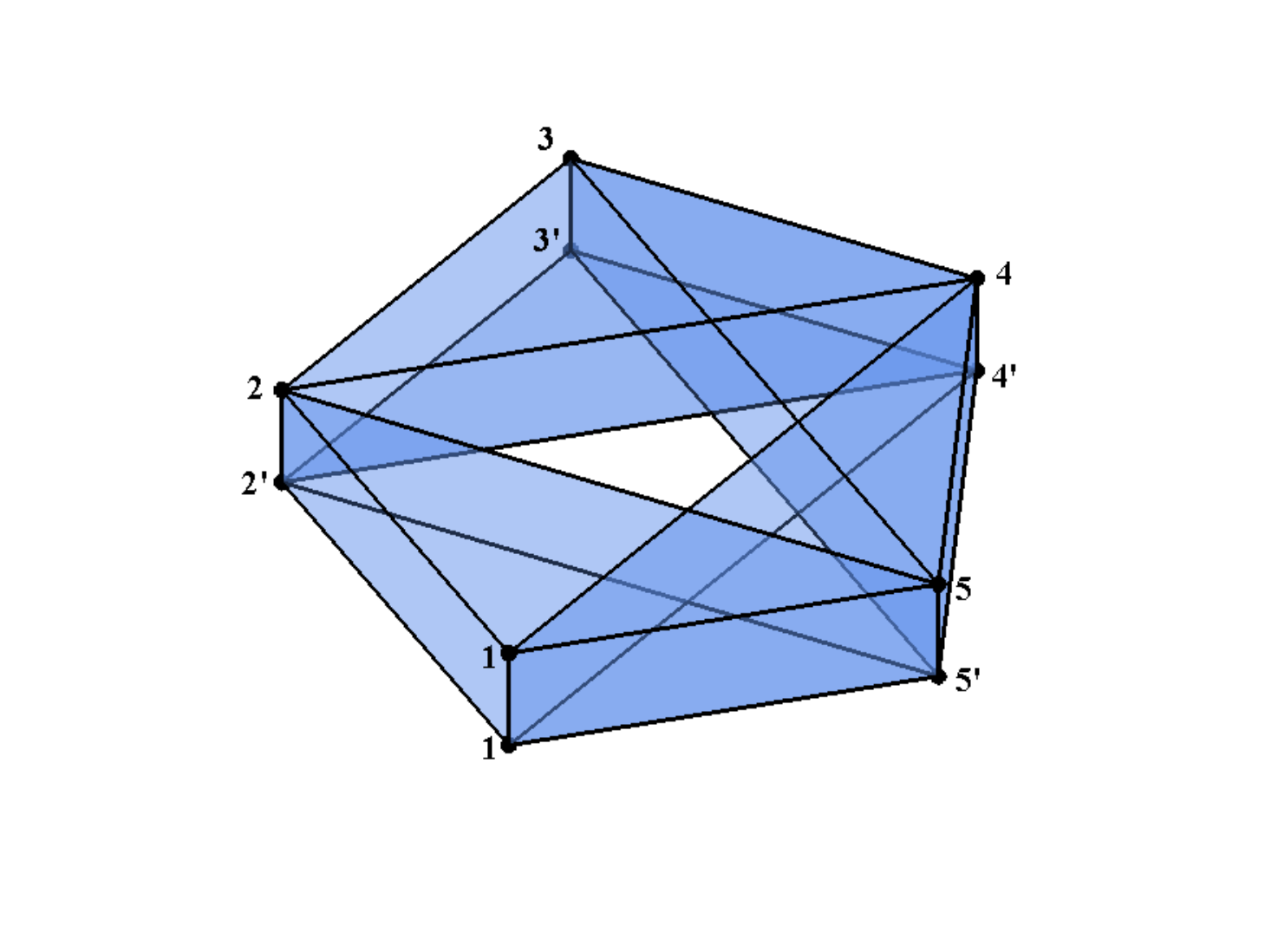}
\caption{The matroid $V_{H}$}
\end{subfigure}
}
\end{figure} \noindent
\begin{proof}
Let 
$$
\mathcal{S}= \{ e\cup e': e \in E(H)\} \cup \left\{ S \in \binom E {2d-1} : S \subset  e\cup e' \mbox{ for no }  e \in E(H) \right\}. 
$$
Then $\mathcal{S}$ is a $(2d-1)$-partition, and so it defines a sparse paving matroid with set of bases $\binom{E}{2d} \setminus \{ e\cup e' : e \in E(H) \}$ by Proposition \ref{d-part}. 
\end{proof}
Let $\mathcal{V} = \{V_H : H \text{ is an  $d$-uniform hypergraph on } [n], 0<d\leq n, n \in \NN\}$.

\begin{example}
If $G$ is the diamond graph (Fig \ref{diamond}) then $V_{G} = V_8$. Moreover $V_{\bar{K}_n} = U_{4, 2n}$, where $\bar{K}_n$ denotes the complement of the complete graph on $n$ vertices and $U_{4,2n}$ denotes the uniform rank $4$ matroid on $2n$ elements. The family $\{V_{2n}\}_{n \geq 4}$ studied by Burton \emph{et al.} \cite{BVY} corresponds to $V_{G_n}$ where $G_n$ is an $n$-cycle with edges $\{1,i\}$, $i=2,\ldots, n$, adjoined. 
\end{example} \noindent
We postpone the proofs of the next two theorems till Section \ref{proofsec}. 
\begin{theorem}\label{genvamoshyp}
All matroids in  $\mathcal{V}$ are hyperbolic, i.e., they all have the weak half-plane property. 
\end{theorem} \noindent
\begin{theorem}\label{genvamoshpp}
For each simple graph $G$,  $V_G$ has the half--plane property. 
\end{theorem} 
If $G$ contains the Diamond graph as an induced subgraph then the rank function of $V_G$ fails to satisfy Ingleton's inequality, and thus $V_G$ is hyperbolic but not modularly represented. 

There is no full analogue of Theorem \ref{genvamoshpp} in the hypergraph setting.  To see this let $H$ be the complete $3$-uniform hypergraph on $[6]$. Then setting $x_1=x_{1'}=t$, $x_2=x_{2'}=x_3=x_{3'}=-2$, and the remaining variables to $1$ in the bases--generating polynomial yields a polynomial in $t$ with non-real zeros. Hence $V_{H}$ does not have the half-plane property. Clearly if $V_8$ is a minor of $V_H$ then $V_H$ cannot be representable.
Below we give an example of a non-representable matroid $V_H$ with no V\'amos minor. Hence this constitutes a genuinely new instance of a hyperbolic matroid in the family which is not representable. 
\begin{example}
The following linear rank inequality in six variables was identified by Dougherty \emph{et al.} \cite{DFZ}
\begin{align*}
\displaystyle &r(A \cup D) + r(B \cup C) + r(C \cup E) + r(E \cup F) + r(B \cup D \cup F) + r(A \cup B \cup C \cup D) + \\ &r(A \cup B \cup C \cup E) + r(A \cup C \cup E \cup F) + r(A \cup D \cup E \cup F)  \leq \\ &r(A \cup B \cup C) + r(A \cup B \cup D) + r(A \cup C \cup E) + r(A \cup D \cup F) + r(A \cup E \cup F) + \\ &r(B \cup C \cup D) + r(B \cup C \cup E) + r(C \cup E \cup F) + r(D \cup E \cup F).
\end{align*} \noindent
This inequality is satisfied by all polymatroids $r$ representable over some field, where $r : 2^{[n]} \rightarrow \NN$, $n \in \NN$ and $A,B,C,D,E, F \subseteq [n]$. 
We proceed by designing a 3-uniform hypergraph $H$ on $[6]$ such that $V_H$ violates the above inequality. Let
\begin{align*}
\displaystyle A = \{1,1'\}, \thickspace B = \{ 2,2' \}, \thickspace C = \{3,3'\}, \thickspace D = \{4,4'\}, \thickspace E = \{5,5'\}, \thickspace F = \{6,6'\}.
\end{align*} \noindent
By taking the hypergraph $H$ with edges 
\begin{align*}
\displaystyle \{1,2,3\}, \{1,2,4\}, \{1,3,5\}, \{1,4,6\},\{1,5,6\}, \{2,3,4\},\{2,3,5\},\{3,5,6\},\{4,5,6\}
\end{align*} \noindent
we see that $V_H$ violates the above inequality. One checks that $V_8$ is not a minor of $V_H$. 
\end{example}
\section{Consequences for the generalized Lax conjecture} \noindent
Helton and Vinnikov \cite{HV} conjectured  that if a polynomial  $h \in \mathbb{R}[x_1, \dots, x_n]$ is hyperbolic with respect to $\ee = (e_1, \dots, e_n) \in \RR^n$, then there exist positive integers $M,N$ and a linear polynomial $\ell(\xx) \in \mathbb{R}[x_1, \dots, x_n]$ which is positive on $\Lambda_{++}(h,\ee)$ such that 
\begin{align*}
\displaystyle \ell(\xx)^{M-1}h(\xx)^N = \det \left ( \sum_{i=1}^n x_i A_i \right )
\end{align*} \noindent
for some symmetric matrices $A_1, \dots, A_n$ such that $e_1A_1 + \cdots + e_nA_n$ is positive definite.
In \cite{B} the second author used the bases generating polynomial $h_{V_8}$ of the V\'amos matroid to prove that there is no linear polynomial $\ell(\xx) \in \mathbb{R}[x_1, \dots, x_n]$ which is nonnegative on the hyperbolicity cone of $h_{V_8}$ and positive integers $M,N$ such that
\begin{align*}
\ell(\xx)^{M-1} h_{V_8}(\xx)^N = \det\left (\sum_{i=1}^8 x_i A_i \right )
\end{align*} \noindent
for some symmetric matrices $A_1, \dots, A_8$ with $e_1A_1 + \cdots + e_8A_8$ positive definite. We will here construct further ``counterexamples" that preclude more general factors $q(\xx)$ in \eqref{qhd}. First we prove two lemmata of matroid theoretic nature. If $r : 2^{E} \rightarrow \NN$ is a polymatroid and $A \subseteq E$, we say that $A$ is \emph{spanning} if $r(A)=r(E)$.  Moreover $A \subset E$ is a \emph{hyperplane} if it is a maximal non--spanning set.  
\begin{lemma} \label{nonspanupbnd}
For $n,r,c \geq 1$, let $\mathcal{P}(n,r,c)$ be the family of all rank at most $r$ polymatroids on $n$ elements such that each hyperplane has at most $r-1+c$ elements. If $\alpha(n,r,c)$ denotes the maximal number of non-spanning sets of size $r$ taken over all matroids in $\mathcal{P}(n,r,c)$, then 
\begin{equation}\label{anrc}
\alpha(n,r,c) \leq c\binom n {r-1}.
\end{equation}
\end{lemma}
\begin{proof}
If $r=1$, then each hyperplane has at most $c$ elements, i.e., there are at most $c$ loops so that 
$\alpha(n,r,c) = c$ as desired. The proof is by induction over $n \geq 1$ where $r \geq 1$. The lemma is trivially true for $n=1$. 

Let $\mathcal{P} \in \mathcal{P}(n,r,c)$, where $n, r \geq 2$. If $n \leq r$, then \eqref{anrc} is trivially true. Assume $n>r$.
Let $i$ be a non-loop of $\mathcal{P}$. If $r(E\setminus i) < r(E)$, then $E\setminus i$ is a hyperplane and hence $n-1 \leq r-1+c$, so that $\binom n r \leq c\binom n {r-1}$. Hence we may assume $r(E\setminus i) = r(E)>0$. 

If $S$ is a non-spanning $r$-set of $\mathcal{P}$, then either $S$ is a non-spanning $r$-set of $\mathcal{P} \setminus i$, or $S\setminus i$ is a non-spanning $(r-1)$-set of $\mathcal{P} /i$. Hence $\mathcal{P} \setminus i \in \mathcal{P}(n-1,r,c)$ and 
$\mathcal{P} / i \in \mathcal{P}(n-1,r-1,c)$, and thus
\begin{align*}
\alpha(n,r,c) &\leq \alpha(n-1,r,c)+ \alpha(n-1,r-1,c) \\
     &\leq c\binom {n-1} {r-1}+c\binom {n-1} {r-2} = c\binom {n} {r-1},
\end{align*}
by induction.      
\end{proof} \noindent
\begin{lemma} \label{upbnd}
Let $\mathcal{P}_i$, $i=1,\ldots, s$, be  polymatroids on $[n]$ of rank at most $k-1$ such that no hyperplane has more than $k$ elements. If $n \geq (2s+1)k-1$, then there is a set $S$ of size $k$ such that there are at least two $(k-1)$-subsets of $S$ that are spanning in all $\mathcal{P}_i$, $i=1,\ldots, s$. 
\end{lemma}
\begin{proof}
Suppose the conclusion is not true. Let 
$$
A=\left\{ (S,T) : \binom {[n]} {k-1} \ni S \subset T \in \binom {[n]} {k}, S \mbox{ is not spanning in } \mathcal{P}_i \mbox{ for some } i \in [s] \right\}. 
$$
Then
$$|A| \geq (k-1) \binom n k.
$$ 
Furthermore by Lemma \ref{nonspanupbnd} we have
\begin{align*}
\displaystyle |A| &= \# \left \{ S \subseteq \binom{n}{k-1} : S \mbox{ is not spanning in } \mathcal{P}_i \mbox{ for some } i \in [s] \right \} \cdot (n-k+1) \\ &\leq s \alpha(n,k-1,2)(n-k+1) \\ & \leq 2s\binom{n}{k-2}(n-k+1).
\end{align*}
Hence 
$$
(k-1) \binom n k \leq 2s\binom{n}{k-2}(n-k+1).
$$
Solving for $n$ gives $n \leq (2s+1)k-2$, which proves the lemma. 
\end{proof} \noindent
Given positive integers $n$ and $k$, consider the $k$-uniform hypergraph $H_{n,k}$ on $[n+2]$ containing all hyperedges $e \in \binom{[n+2]}{k}$ except those for which $\{n+1, n+2\} \subseteq e$. By Theorem \ref{genvamoshyp} the matroid $V_{H_{n,k}}$ is hyperbolic and therefore has a stable weighted bases generating polynomial $h_{V_{H_{n,k}}}$ by Proposition \ref{whpphyp}. The polynomial $h_{n,k} \in \mathbb{R}[x_1, \dots, x_{n+2}]$ obtained from the multiaffine polynomial $h_{V_{H_{n,k}}}$ by identifying the variables $x_i$ and $x_{i'}$ pairwise for all $i \in [n+2]$ is stable. Hence by Lemma \ref{hypbas} we have $\mathbb{R}_+^{n+2} \subseteq \Lambda_+(h_{n,k})$ so $h_{n,k}$ is hyperbolic with respect to $\mathbf{1}$. 
\begin{theorem} \label{counterex}
Let $n$ and $k$ be  a positive integers. Suppose there exists a positive integer $N$ and a hyperbolic polynomial $q(\xx)$ such that
\begin{align}\label{qhnk}
\displaystyle q(\mathbf{x}) h_{n,k}(\mathbf{x})^N = \det \left ( \sum_{i=1}^{n+2} x_i A_i \right )
\end{align} \noindent
with $\Lambda_+(h_{n,k}) \subseteq \Lambda_+(q)$ for some symmetric matrices $A_1, \dots, A_{n+2}$ such that $A_1 + \cdots + A_{n+2}$ is positive definite and
\begin{align*}
\displaystyle q(\mathbf{x}) = \prod_{i=1}^s p_j(\mathbf{x})^{\alpha_i}
\end{align*} \noindent
for some irreducible hyperbolic polynomials $p_1, \dots, p_s \in \mathbb{R}[x_1, \dots, x_{n+2}]$ of degree at most $k-1$ where $\alpha_1,\dots, \alpha_s$ are positive integers. Then
\begin{align*}
\displaystyle n < (2s+1)k-1.
\end{align*} \noindent
\end{theorem}
\begin{proof}
Suppose the hypotheses are satisfied and $n \geq (2s+1)k-1$.
Let $r_0 : 2^{[n]} \rightarrow \NN$ be the hyperbolic polymatroid defined by $h_{n,k}$ and $\VV=(\delta_1,\ldots, \delta_n)$, where
$\delta_i$, $i \in [n]$ are the standard basis vectors. Hence 
$r_0(S)$ is the rank of $S \cup \{i': i \in S\}$ in the matroid $V_{H_{n,k}}$. Moreover, let $r_i : 2^{[n]} \rightarrow \NN$, $i \in [s]$, be the hyperbolic polymatroid defined by $p_i$ and $\VV=(\delta_1,\ldots, \delta_n)$. Any subset $S$ of $[n]$ of size at least $k+1$ is spanning for  $r_0$, and thus 
$\sum_{i \in S} \delta_i \in \Lambda_{++}(h_{n,k})$. Hence $\sum_{i \in S} \delta_i \in \Lambda_{++}(p_i)$, and thus $S$ is spanning with respect to $r_i$ for all $i \in [s]$. 
By Lemma \ref{upbnd}, since $n \geq (2s+1)k-1$, there exists a subset $T \subseteq [n]$ of size $k$ containing at least $2$ distinct subsets $X,Y$ of size $k-1$ with full rank with respect to all hyperbolic polymatroids $r_{i}$, $i = 1,\dots, s$.  Let $x,y \in T$ be the unique elements in $X,Y$, respectively, not contained in $Z = X \cap Y$. Define
$$
A=Z \cup \{n+1\}, \quad B=Z \cup \{n+2\}, \quad C=Z \cup \{x\}, \quad D=Z \cup \{y\}. 
$$ 
Now $A\cup B, A\cup C \cup D$ and $B\cup C \cup D$ have full rank with respect to $r_0$. Since $\Lambda_{++}(h_{n,k}) \subseteq \Lambda_{++}(p_i)$ for all $i$, we see that  $A\cup B, A\cup C \cup D$ and $B\cup C \cup D$ have full rank with respect to $r_i$ for all $i$. Hence the rank of each set to the left in the Ingleton inequality have full rank with respect to $r_{i}$, so that 
\begin{align*}
\displaystyle & r_{i}(A \cup B) + r_{i}(A \cup C \cup D) + r_{i}(C) + r_{i}(D) + r_{i}(B \cup C \cup D) \geq \\ & r_{i}(A \cup C) + r_{i}(A \cup D) + r_{i}(B \cup C) + r_{i}(B \cup D) + r_{i}(C \cup D)
\end{align*} \noindent
for $i = 1,\dots, s$. 
Note also that 
\begin{align*}
\displaystyle &r_{0}(A \cup B) + r_{0}(A \cup C \cup D) + r_{0}(C) + r_{0}(D) + r_{0}(B \cup C \cup D) = \\ &2k + 2k + (2k-2) + (2k-2) + 2k > (2k-1) + (2k-1)+ (2k-1) + (2k-1)+ (2k-1) = \\ &r_{0}(A \cup C) + r_{0}(A \cup D) + r_{0}(B \cup C) + r_{0}(B \cup D) + r_{0}(C \cup D).
\end{align*} \noindent
Thus $r_{0}$ violates the Ingleton inequality. 
Let $\mathcal{R}$ denote the representable polymatroid with rank function
\begin{align*}
\displaystyle r_{\mathcal{R}}(S) = \text{rank} \left ( \sum_{i \in S} A_i \right ).
\end{align*} \noindent
for all $S \subseteq [n]$. Then, by \eqref{qhnk}, 
$$
 r_{\mathcal{R}}(S) = \text{rank} \left ( \sum_{i \in S} A_i \right ) = \sum_{i = 1}^s \alpha_i r_{i}(S) + N r_0(S).
$$
 \noindent
Hence $r_{\mathcal{R}}$ violates Ingleton's inequality, a contradiction.
\end{proof} \noindent
Hence, for $n$ sufficiently large, $q$ in \eqref{qhd} either has an irreducible factor of large degree or is the product of many factors of low degree. 

Consider  
$$
h_{2,2} = x_1^2x_2^2 + 4(x_1+x_2+x_3+x_4)(x_1x_2x_3 + x_1x_2x_4 + x_1x_3x_4 + x_2x_3x_4).
$$
The polynomial $h_{2,2}$ comes from the bases generating polynomial of the V\'amos matroid under the restriction $x_i = x_{i'}$ for $i = 1, \dots, 4$. Kummer \cite{Kum} found real symmetric matrices $A_i$, $i = 1,\dots, 4$ with $A_1+ A_2+A_3+A_4$ positive definite and a hyperbolic polynomial $q$ of degree $3$ with $\Lambda_+(h_{2,2}) \subseteq \Lambda_+(q)$ such that
\begin{align*}
\displaystyle q(\xx)h_{2,2}(\xx) = \det \left ( x_1A_1 + x_2A_2 + x_3A_3 + x_4A_4 \right ),
\end{align*} \noindent
where 
\begin{align*}
\displaystyle q(\xx) = 32(2x_1 + 3x_2 + 3x_3 + 4x_4)(x_1x_2 + x_1x_3 + 2x_1x_4 + x_2x_4 + x_3x_4).
\end{align*} \noindent
If $s = 2$ and $k=3$ in Theorem \ref{counterex} it follows that there exists no linear and quadratic hyperbolic polynomials $\ell(\xx), q(\xx) \in \mathbb{R}[x_1, \dots, x_{16}]$ respectively such that 
$h_{13,3}(\xx) \in \mathbb{R}[x_1, \dots, x_{15}]$ has a positive definite representation of the form
\begin{align*}
\displaystyle \ell(\xx)q(\xx)h_{14,3}(\xx) = \det \left ( \sum_{i=1}^{16} x_iA_i \right )
\end{align*} \noindent
with $\Lambda_+(h_{14,3}) \subseteq \Lambda_+(\ell q)$.

\section{Nonnegative symmetric polynomials}
Recall that a polynomial $P(\xx) \in \RR[x_1,\ldots, x_n]$ is \emph{nonnegative} if $P(\xx) \geq 0$ for all $\xx \in \RR^n$, and it is \emph{symmetric} if it is invariant under the action (permuting the variables) of the symmetric group of order $n$. In this section we prove that certain symmetric polynomials are nonnegative. This is needed for the proof of Theorem \ref{genvamoshyp}.  The results are interesting in their own right, and they generalize several well known inequalities in the literature. 

Recall that a \emph{partition} of a natural number $d$ is a sequence $\lambda =(\lambda_1, \lambda_2, \ldots )$ of natural numbers such that $\lambda_1 \geq \lambda_2 \geq \cdots$ and $\lambda_1+\lambda_2+\cdots = d$. We write $\lambda \vdash d$ to denote that $\lambda$ is a partition of $d$. The \emph{length}, $\ell(\lambda)$, of $\lambda$ is the number of nonzero entries of $\lambda$. If $\lambda$ is a partition and $\ell(\lambda) \leq n$, then  the \emph{monomial symmetric polynomial}, $m_\alpha$, is defined as 
$$
m_\lambda(x_1,\ldots, x_n)=\sum x_1^{\beta_1}x_2^{\beta_2} \cdots x_n^{\beta_n}, 
$$
 where the sum is over all 
distinct permutations $(\beta_1, \beta_2, \ldots, \beta_n)$ of $(\lambda_1, \ldots, \lambda_n)$. If $\ell(\lambda)>n$, we set $m_\lambda(\xx) =0$. If $k_1,\ldots, k_\ell$ are distinct positive integers and 
$a_1, \ldots, a_\ell \in \NN$ we denote by $k_1^{a_1}k_2^{a_2} \cdots k_{\ell}^{a_\ell}$ the unique partition of  $a_1+\cdots + a_\ell$  with exactly $a_j$ coordinates equal to $k_j$ for $1 \leq j \leq \ell$. The $d$th \emph{elementary symmetric polynomial} is 
$
e_d(\xx)= m_{1^d}(\xx), 
$
and the $d$th \emph{power symmetric polynomial} is $p_d(\xx)= m_{d}(\xx)$.

Nonnegative symmetric polynomials  have been studied in several areas of mathematics, see \cite{BlRie,CL2,Foster-Krasikov} and the references therein. We will initially concentrate on nonnegative polynomials of the form 
\begin{equation}\label{2ek}
\sum_{k=0}^{2r}a_k e_k(\xx) e_{2r-k}(\xx), \quad \xx \in \RR^m, 
\end{equation}
where $r$ is a positive integer and $\{a_k\}_{k=0}^{2r} \subset \RR$. Hence these are the nonnegative symmetric polynomials spanned by $\{m_{2^k1^{2(r-k)}} : 0 \leq k \leq r\}$.  
A classical family of such nonnegative and symmetric polynomials was found already by Newton \cite{Newton}:
$$
\frac {e_r(\xx)^2}{{\binom n r}^2} - \frac {e_{r-1}(\xx)}{\binom n {r-1}} \frac {e_{r+1}(\xx)}{\binom n {r+1}} \geq 0, \quad \xx \in \RR^m, m \leq n, 1\leq r \leq n-1. 
$$
Letting $n \to \infty$ in Newton's inequalities we obtain the Laguerre--Tur\'an inequalities (see e.g. \cite{CC}):
$$
r e_r(\xx)^2 - (r+1)  e_{r-1}(\xx)e_{r+1}(\xx) \geq 0, \quad \xx \in \RR^m, m \geq 1. 
$$
A different but equivalent view on nonnegative symmetric polynomial is that of inequalities satisfied by the derivatives of a real--rooted polynomial: Let $\{a_k\}_{k=0}^{2m}$ be a sequence of real numbers. Then the polynomial \eqref{2ek}  is  nonnegative if and only if 
\begin{equation}\label{2fk}
\sum_{k=0}^{2r} a_k  \binom {2r} k  f^{(k)}(t) f^{(2r-k)}(t) \geq 0, \quad t \in \RR,  
\end{equation}
holds for all real--rooted polynomials $f$ of degree at most $m$. Indeed by translation invariance  \eqref{2fk} holds for all 
real--rooted polynomials $f$ of degree at most $m$ if and only if \eqref{2fk} holds at $t=0$ for all real--rooted polynomials $f$ of degree at most $m$. Hence if 
$f(t)= \prod_{j=1}^m (1+x_jt)$, then the left--hand--side of \eqref{2fk} at $t=0$ is the same as \eqref{2ek} up to a constant factor $(2r)!$. The following inequalities are due to Jensen \cite{Jensen}:
\begin{equation}\label{jens}
\sum_{k=0}^{2r} (-1)^{r+j} \binom {2r} k   f^{(k)}(t) f^{(2r-k)}(t) \geq 0, \quad t \in \RR,  
\end{equation}
for all real--rooted polynomials $f$. Jensen's inequality follows easily from a symmetric function identity as follows 
\begin{align*}
\sum_{r=0}^n m_{2^r}(\xx)t^{2r} &= \sum_{r=0}^n e_{r}(x_1^2,\ldots,x_n^2)t^{2r} \\
&= \prod_{j=1}^n(1+x_j^2t^2) = \prod_{j=1}^n(1+ix_jt) \prod_{j=1}^n(1-ix_jt) \\
&= \left(\sum_{k=0}^n i^ke_{k}(\xx)t^{k}\right) \left(\sum_{k=0}^n (-i)^ke_{k}(\xx)t^{k}\right) \\
&= \sum_{r=0}^n \left( \sum_{k=0}^{2r} (-1)^{k+r} e_k(\xx)e_{2r-k}(\xx) \right) t^{2r}.
\end{align*} \noindent
Clearly $m_{2r}(\xx) \geq 0$ for all $\xx \in \RR^n$, so that Jensen's inequality follows from 
\begin{equation}\label{jensid}
m_{2^r}(\xx) =  \sum_{k=0}^{2r} (-1)^{k+r} e_k(\xx)e_{2r-k}(\xx).
\end{equation}

\begin{lemma}\label{tsos}
If $r$ is a positive integer and $0 \leq t \leq 2/r$, then 
$$
m_{2^r}(\xx)+ t m_{2^{r-1}1^2}(\xx)
$$
is a sum of squares (sos for short), and in particular nonnegative. 
\end{lemma}
\begin{proof}
Since $m_{2^r}(\xx)$ is a sos it suffices to consider $t=2/r$, by convexity. 
Note that 
$$
m_{2^{r-1}1^2}(\xx) = \sum_{|S|=r-1}e_2(\xx^{S})\prod_{i  \in S}x_i^2 ,
$$
where $\xx^{S}=\xx \setminus \{x_i : i \in S\}$. Using $e_2(\xx)= e_1(\xx)^2/2-p_2(\xx)/2$ 
\begin{align*}
m_{2^{r-1}1^2}(\xx) &= \frac 1 2 \sum_{|S|=r-1}e_1(\xx^{S})^2\prod_{i  \in S}x_i^2-\frac 1 2 \sum_{|S|=r-1}p_2(\xx^{S})\prod_{i \in S}x_i^2 \\
&= S(\xx)- \frac r 2 m_{2^r}(\xx),
\end{align*}
where $S(\xx)$ is a sum of squares. Indeed 
$$
\frac 1 2 \sum_{|S|=r-1}p_2(\xx^{S})\prod_{i  \in S}x_i^2 =C m_{2^r}(\xx)
$$
for some $C$, and setting $\xx=(1,\ldots,1)$ 
$$
\frac 1 2 \binom n {r-1} (n-r+1)= C \binom n r, 
$$
so that $C= r/2$. The lemma follows. 
\end{proof}

Let $P(\xx)$ be a symmetric polynomial. Suppose 
$P(\xx)= Q(e_1(\xx), \ldots, e_m(\xx))$ is the unique expression of $P$ in terms of the elementary symmetric polynomials. If $Q$ is of degree $d$, let $H(x_0,x_1,\ldots,x_m)=x_0^dQ(x_1/x_0,\ldots, x_m/x_0)$ be its homogenization, and let 
$$
L(P) := H(e_1(\xx), 2e_2(\xx), \ldots, (m+1)e_{m+1}(\xx))
$$
be the \emph{lift} of $P$. 

\begin{lemma}\label{lift}
If $P$ is a symmetric and nonnegative polynomial, then so is its lift $L(P)$.
\end{lemma}

\begin{proof}
Note first that if $P$ is nonnegative and symmetric, then the degree of $Q$ above is even. Indeed if $\xx(t) = (t,x_2,\ldots, x_n)$ where $x_2,\ldots, x_n \in \RR$ are generic and $t$ is a variable, then we obtain a univariate nonnegative polynomial $t \mapsto P(\xx(t))$ of degree $d$. Hence $d$ is even. Now if $\xx \in \RR^n$ is such that $e_1(\xx) \neq 0$, then there is a $\yy \in \RR^{n}$ such that  $e_k(\yy) = (k+1)e_{k+1}(\xx)/e_1(\xx)$ for all $k$. Indeed 
$$
\frac 1 {e_1(\xx)} \frac d {dt} \prod_{k=1}^n (1+x_kt) = \prod_{k=0}^{n} (1+y_kt), \ \ \ \mbox{ where } y_n=0, 
$$
since the operator $d/dt$ preserves real--rootedness. Thus 
$$
L(P)(\xx)= P(\yy)
$$
and the proof follows.

\end{proof}

\begin{lemma}\label{boost}
The lift of $m_{2^{r-1}}(\xx)$ is 
$$
r^2 m_{2^r}(\xx) + 2m_{2^{r-1}1^2}(\xx).
$$
\end{lemma}

\begin{proof}
By \eqref{jensid}, the lift of 
$
m_{2^{r-1}}(\xx)
$
is 
\begin{align*}
L&:=\sum_{j=0}^{2r-2}(-1)^{j+r-1}(j+1)(2r-1-j)e_{j+1}(\xx)e_{2r-1-j}(\xx) \\
&=\sum_{j=0}^{2r}(-1)^{j+r}j(2r-j)e_{j}(\xx)e_{2r-j}(\xx). 
\end{align*}

The coefficient infront of $m_{2^{k}1^{2(r-k)}}(\xx)$ in the expansion of $e_j(\xx)e_{2r-j}(\xx)$ in the monomial bases is seen to be $\binom {2r-2k}{j-k}$. (Look at  how many times we get the monomial $x_1^2x_2^2 \cdots x_k^2 x_{k+1}x_{k+2} \cdots$ in the expansion of the $e_j(\xx)e_{2r-j}(\xx)$.) Hence the coefficient infront of $m_{2^{k}1^{2(r-k)}}(\xx)$ in the expansion of $L$ in the monomial basis is 
$$
a_k=\sum_{j=0}^{2r}(-1)^{j+r}j(2r-j)\binom {2r-2k}{j-k}.
$$
Now $a_r=r^2, a_{r-1}=2$, and $a_k=0$ otherwise. This follows from the fact if $p$ is a polynomial of degree $d$, then  
$$
\sum_{j=0}^n(-1)^jp(j)\binom n j =0
$$ 
whenever $n>d$.
\end{proof}

Our next lemma is a refinement of the Laguerre--Tur\'an inequalities and may be formulated as \emph{the 
Laguerre--Tur\'an inequalities beat Jensen's inequalities}. Lemma \ref{turan} is also a generalization of 
\cite[Theorem 3]{Foster-Krasikov}, where the case $r=2$ was proved. If $P, Q \in \RR[\xx]$, we write 
$P \leq Q$ if $Q -P$ is a nonnegative polynomial. 
\begin{lemma}\label{turan}
If $r \geq 1$, then 
$$
m_{2^r}(\xx) \leq r e_r(\xx)^2 - (r+1)  e_{r-1}(\xx)e_{r+1}(\xx).
$$
\end{lemma}

\begin{proof}
The proof is by induction over $r$. For $r=1$ we have equality. Assume true for $r-1$ where  $r\geq 2$, and lift the inequality. By Lemma \ref{lift} and Lemma \ref{boost} 
$$
r^2 m_{2^r}(\xx) + 2m_{2^{r-1}1^2}(\xx) \leq  
 (r-1) r^2 e_r(\xx)^2 - r(r-1)(r+1) e_{r-1}(\xx)e_{r+1}(\xx)
$$
By Lemma \ref{tsos} 
$$
(r^2-r)m_{2^r}(\xx) \leq r^2 m_{2^r}(\xx) + 2m_{2^{r-1}1^2}(\xx),
$$
and the lemma follows. 
\end{proof}

\begin{lemma}\label{eng}
If $r \geq 2$ is an integer, then 
\begin{equation}\label{aC}
(a_re_{r-1}(\xx) e_r(\xx)-e_{r-2}(\xx) e_{r+1}(\xx))^2 \geq C_r e_{r-2}(\xx)e_r(\xx) m_{2^r}(\xx),
\end{equation}
where 
$$
a_r= 3 \frac {r-1} {r+1} \ \ \mbox{ and } \ \ C_r= 9 \frac {r-1}{(r+1)^2}.
$$
\end{lemma}
\begin{proof}
We prove the inequality by induction over $r \geq 2$. 
Assume $r=2$. The polynomial $(t,\xx) \mapsto e_4(t,t,x_1,x_1,\ldots,x_n,x_n)$ is stable. It specializes to a real-rooted (or identically zero) polynomial when we set $\xx = (x_1,\ldots, x_n) \in \RR^n$: 
$$
(4e_2(\xx)+p_2(\xx))t^2+4(e_1(\xx)e_2(\xx)+e_3(\xx))t+m_{2^2}(\xx)+4e_1(\xx)e_3(\xx). 
$$
Hence its discriminant is nonnegative, which gives 
$$
(e_3(\xx)+e_1(\xx)e_2(\xx))^2 \geq (e_2(\xx)+ p_2(\xx)/4)(m_{2^2}(\xx)+4e_1(\xx)e_3(\xx)). 
$$
To prove \eqref{aC} for $r=2$ we may assume $e_2(\xx) >0$. Rewriting \eqref{aC} as 
$$
(e_3(\xx)+e_1(\xx)e_2(\xx))^2 \geq e_2(\xx)(m_{2^2}(\xx)+4e_1(\xx)e_3(\xx)), 
$$
we may assume also $m_{2^2}(\xx)+4e_1(\xx)e_3(\xx) >0$. Then, since $p_2(\xx) \geq 0$, 
\begin{align*}
(e_3(\xx)+e_1(\xx)e_2(\xx))^2 &\geq (e_2(\xx)+ p_2(\xx)/4)(m_{2^2}(\xx)+4e_1(\xx)e_3(\xx)) \\
 &\geq e_2(\xx)(m_{2^2}(\xx)+4e_1(\xx)e_3(\xx))
\end{align*}
which proves the lemma for $r=2$.

Assume true for a given $r\geq 2$. We lift the inequality for $r$ and use Lemma \ref{boost} to get 
\begin{align*}
&(a_rr(r+1)e_r(\xx)e_{r+1}(\xx)-(r-1)(r+2)e_{r-1}(\xx)e_{r+2}(\xx))^2 \geq \\
& C_r (r-1)(r+1)^3e_{r-1}(\xx)e_{r+1}(\xx) \left(m_{2^{r+1}(\xx)}+ \frac 2 {(r+1)^2}m_{2^r1^2}(\xx)\right)
\end{align*}
We may exchange the factor $m_{2^{r+1}}+ (2/ {(r+1)^2})m_{2^r1^2}$ by something nonnegative and smaller and still get a valid inequality. By Lemma \ref{tsos} we obtain the inequality 
\begin{align*}
&(a_rr(r+1)e_r(\xx)e_{r+1}(\xx)-(r-1)(r+2)e_{r-1}(\xx)e_{r+2}(\xx))^2 \geq \\
& C_r (r-1)(r+1)^3e_{r-1}(\xx)e_{r+1}(\xx) \frac r {r+1} m_{2^{r+1}}(\xx). 
\end{align*}
Dividing through by $(r-1)^2(r+2)^2$ we obtain 
\begin{align*}
&\left(a_r \frac {r(r+1)}{(r-1)(r+2)}e_r(\xx)e_{r+1}(\xx)-e_{r-1}(\xx)e_{r+2}(\xx)\right)^2 \geq \\
& C_r \frac {r(r+1)^2}{(r-1)(r+2)^2}e_{r-1}(\xx)e_{r+1}(\xx)m_{2^{r+1}}(\xx), 
\end{align*}
which simplifies to the desired inequality for $r+1$. 
\end{proof}

\section{Proof of Theorem \ref{genvamoshpp}}\label{proofsec}\noindent 
The next tool for the proof of Theorem \ref{genvamoshpp} is a lemma that enables us to prove hyperbolicity of a polynomial by proving real-rootedness along a few (degenerate) directions. 
\begin{lemma}\label{garden}
Let $h \in \CC[x_1,\ldots, x_n]$  and  $\vv_1, \vv_2 \in \CC^n$. Define $r$ be the maximum degree of the polynomial $t \mapsto h(t\vv_2+\yy)$,  where the maximum is taken over all $\yy \in \CC^n$. Let further 
$$
 P(\xx) := \lim_{t \rightarrow \infty} t^{-r}h(t\vv_2+\xx) \in \CC[x_1,\ldots, x_n].
$$
Suppose $S \subseteq \CC^n$ and  $\xx_0 \in S$ are such that
\begin{enumerate}
\item $S+\RR \vv_2 =S$. 
\item For each $\xx_1 \in S$ there is a continuous path $\xx(\theta) : [0,1] \rightarrow S$  such that $\xx(0)=\xx_0$ and $\xx(1)=\xx_1$.
\item The polynomial $(s,t) \mapsto h(s\vv_1+t\vv_2+\xx_0)$ is  stable and not identically zero.
\item For each $\xx \in S$, the polynomials $s \mapsto h(s\vv_1+\xx)$ and $s \mapsto P(s\vv_1+\xx)$ are stable and not identically zero. 
\end{enumerate}
Then the polynomial $(s,t) \mapsto h(s\vv_1+t\vv_2+\xx)$ is stable for all  $\xx \in S$.  
\end{lemma}
\begin{proof}
The proof is by contradiction. Suppose $\xx_1 \in S$ and $\xi, \eta \in \CC$ are such that 
$\Im(\xi)>0$, $\Im(\eta) > 0$ and 
$$
h(\xi \vv_1+\eta \vv_2 + \xx_1)=0.
$$
Let $\xx(\theta) : [0,1] \rightarrow S$ be a continuous path such that $\xx(0)=\xx_0$ and $\xx(1)=\xx_1$ and let 
$$
p_\theta(t)= h(\xi \vv_1+t \vv_2 + \xx(\theta))= t^r P(\xi \vv_1+ \xx(\theta)) + O(t^{r-1}), 
$$
where $P(\xi \vv_1+ \xx(\theta)) \neq 0$ by (iv). 
By assumption all zeros of $p_0(t)$ are in the closed lower half-plane, while $p_1(\eta)=0$ where $\Im(\eta) > 0$. Hence, by continuity, a zero will cross the real axis as $\theta$ runs from $0$ to $1$. In other words 
$$
0= p_\theta ( \alpha) = h(\xi \vv_1+\alpha \vv_2 + \xx(\theta)),
$$
for some $\alpha \in \RR$ and $\theta \in [0,1]$. Since $\alpha \vv_2 + \xx(\theta) \in S$, by (i), this contradicts (iv). 
\end{proof}

The next theorem is a version of the Grace-Walsh-Szeg\H{o} coincidence theorem, see \cite[Prop. 3.4]{BB}.  
\begin{theorem}[Grace-Walsh-Szeg\H{o}] 
Suppose $P(x_1,\ldots, x_n) \in \CC[\xx]$ is a  polynomial of degree at most $d$ in the variable $x_1$:
$$
P(x_1,\ldots, x_n)= \sum_{k=0}^d P_k(x_2,\ldots, x_n)x_1^k. 
$$
Let $Q$ be the polynomial in the variables $x_2,\ldots, x_n, y_1,\ldots, y_n$ 
$$
Q=\sum_{k=0}^d P_k(x_2,\ldots, x_n)\frac {e_k(y_1, \ldots, y_d)} { \binom d k }. 
$$
Then $P$ is stable if and only if $Q$ is stable. 
\end{theorem}
\begin{remark}
\label{elemhyp}
Note that $e_d(\mathbf{x})$ is stable, by e.g.  the Grace--Walsh--Szeg\H{o} theorem.
\end{remark}
The following theorem provides families of stable polynomials which are closed under convex sums.
\begin{theorem}\label{basegenstab}
Let  $r \geq 2$ be an integer, and let 
$$
M(\xx)= \sum_{|S|=r} a(S) \prod_{i \in S} x_i^2 \in \RR[x_1,\ldots,x_n], 
$$
where $0 \leq a(S) \leq 1$ for all $S \subseteq [n]$, where $|S|=r$. 
Then the polynomial  
$$
4e_{r+1}(\xx)e_{r-1}(\xx) + \frac 3 {r+1} M(\xx) 
$$
is stable. 
\end{theorem} 

\begin{proof}
We first prove the theorem for the special case when no $x_i$, $i=1,\ldots, 2r+3$ appears in $M$, and $n$ is sufficiently large. 
Let $\xx'=(x_2,\ldots,x_n)$ and 
$$
h(\xx)= 4(a_rx_1e_r(\xx')+e_{r+1}(\xx'))(x_1e_{r-2}(\xx')+e_{r-1}(\xx'))+ \frac 3 {r+1}M.
$$ 
where $a_r$ is defined as in Lemma \ref{eng}. Recall the notation of Lemma \ref{garden}. Let $\vv_1 = \delta_1$, $\vv_2= \delta_2 +\delta_3+\cdots+ \delta_{r+2}$, and let $S$ be the set of all $\xx \in \RR^n$ such that at least $r+1$ of the coordinates $\{x_{r+3}, \ldots, x_n\}$ are nonzero. Let $\xx_0= \delta_{r+3}+ \cdots + \delta_{2r+3}$. Then since 
$$
q(\xx)= 4(a_rx_1e_r(\xx')+e_{r+1}(\xx'))(x_1e_{r-2}(\xx')+e_{r-1}(\xx'))
$$
is stable we now that $h(s\vv_1+t\vv_2+ \xx_0)=q(s\vv_1+t\vv_2+ \xx_0)$ is stable. Note that $P(\xx)$ is a non-zero constant. Consider
\begin{align*}
h(s\vv_1+\xx) &= 4a_re_r(\xx')e_{r-2}(\xx')(s+x_1)^2 \\&+4(a_re_r(\xx')e_{r-1}(\xx')+e_{r+1}(\xx')e_{r-2}(\xx'))(s+x_1)\\
&+ 4e_{r+1}(\xx')e_{r-1}(\xx') + \frac 3 {r+1}M.
\end{align*}
We first prove that $h(s\vv_1+\xx) \not \equiv 0$ for $\xx \in S$. 
Assume $h(s\vv_1+\xx) \equiv 0$ for some $\xx \in S$. Then $e_r(\xx')e_{r-2}(\xx')=0$, so suppose first $e_r(\xx')=0$. If $e_{r+1}(\xx')e_{r-1}(\xx')=0$, then either $e_{r-1}(\xx')=e_r(\xx')=0$ or $e_{r+1}(\xx')=e_r(\xx')=0$, which implies $\xx'$ has at most $r-1$ non-zero coordinates, which contradicts $\xx \in S$. Hence $e_{r+1}(\xx')e_{r-1}(\xx')<0$ by Lemma \ref{turan}. 
Then the constant term satisfies 
\begin{align*}
4e_{r+1}(\xx')e_{r-1}(\xx') + \frac 3 {r+1}M &< 3e_{r+1}(\xx')e_{r-1}(\xx') + \frac 3 {r+1}M  \\
&\leq 3e_{r+1}(\xx')e_{r-1}(\xx') + \frac 3 {r+1}m_{2^r}(\xx') \leq 0, 
\end{align*}
by Lemma \ref{turan}, a contradiction. If $e_r(\xx') \neq 0$, then $e_{r-2}(\xx')=e_{r-1}(\xx')=0$ and hence $\xx$ has at most $r-3$ non-zero coordinates which contradicts $\xx \in S$.  We conclude that $h(s\vv_1+\xx) \not \equiv 0$ for $\xx \in S$. 

To prove that $h(s\vv_1+t\vv_2+\xx)$ is stable for all $\xx \in S$ it remains to prove that $h(s\vv_1+\xx)$ is real--rooted. However $h(s\vv_1+\xx)$ is of degree at most two so it suffices to show that its discriminant $\Delta$ is nonnegative. Now 
\begin{align*}
\frac \Delta {16} &= (a_re_{r-1}(\xx') e_r(\xx')-e_{r-2}(\xx') e_{r+1}(\xx'))^2 - \frac 3 {r+1} a_re_r(\xx')e_{r-2}(\xx')M.
\end{align*}
If $e_r(\xx')e_{r-2}(\xx')<0$, then clearly $\Delta \geq 0$, so assume $e_r(\xx')e_{r-2}(\xx') \geq 0$. Then, since $M(\xx') \leq m_{2^r}(\xx')$, it follows that $\Delta \geq 0$
by Lemma \ref{eng}. 

Since $S$ is dense in $\RR^n$ we have by Hurwitz' theorem that 
 $h(s\vv_1+t\vv_2+\xx)$ is stable or identically zero for all $\xx \in \RR^n$. However 
 $h(\vv_2) \neq 0$ so that $h(s\vv_1+t\vv_2+\xx)$ is stable for all $\xx \in \RR^n$. In particular $h$ is hyperbolic with respect to $\vv_2$. Since all Taylor coefficients of $h$ are nonnegative we have that $h$ is stable, by Lemma \ref{hypbas}. 

The theorem follows in full generality from the special case by setting $x_1 = x_2 = \cdots =x_{2r+3}=0$ in $h$, and relabeling the variables. 

\end{proof}

\begin{lemma}
\label{elemsymcalc}
Let $r \geq 2$. Then
\begin{align*}
&e_{2r}(x_1, x_1, \dots, x_n, x_n) - m_{2^r}(\xx) = \\
&4 (e_{r-1}(\xx)e_{r+1}(\xx)+ e_{r-3}(\xx)e_{r+3}(\xx)+e_{r-5}(\xx)e_{r+5}(\xx)+\cdots).
\end{align*}
\end{lemma}
\begin{proof}
Note that 
\begin{align*}
\displaystyle \sum_{k=0}^{2n} e_k(x_1,x_1, \dots, x_n, x_n)t^k = \prod_{j= 1}^n (1 + x_jt)^2 = \left ( \sum_{k=0}^{2n} e_k(\mathbf{x}) t^k \right )^2.
\end{align*} \noindent
The coefficient of $t^{2r}$ is
\begin{align*}
\displaystyle e_{2r}(x_1, x_1, \dots, x_n, x_n) = \sum_{j=0}^{2r}  e_j(\mathbf{x})e_{2r - j}(\mathbf{x}). 
\end{align*} \noindent
The proof follows by combining this with \eqref{jensid}. 
\end{proof}
We are now in a position to prove Theorem \ref{genvamoshpp} and Theorem \ref{genvamoshyp}. 

\begin{proof}[Proof of Theorem \ref{genvamoshyp}]
By definition the bases generating polynomial of $V_H \in \mathcal{V}$ is given by
\begin{align*}
\displaystyle h_{V_H} = \sum_{B \in \mathcal{B}(V_H)} \prod_{i \in B} x_{i} = e_{2r}(x_1, x_{1'}, \dots, x_n, x_{n'}) - e_r(x_1x_{1'}, \dots, x_n x_{n'}) + N(\xx).
\end{align*} \noindent
where 
\begin{align*}
N(\xx) = \sum_{(i_1, \dots, i_r) \not \in E(H) } \prod_{j = 1}^r x_{i_j}x_{i_j'}.
\end{align*} \noindent
The polynomial $h_{V_H}$ is clearly multiaffine and symmetric pairwise in $x_i, x_{i'}$ for all $i \in [n]$. Set $x_{i'}=x_i$ for all $1 \leq i \leq n$ and obtain the polynomial 
\begin{align*}
\displaystyle f_{V_H} = e_{2r}(x_1, x_1, \dots, x_n, x_n) - e_r(x_1^2, \dots, x_n^2) + N(x_1, x_1, \dots, x_n, x_n).
\end{align*} \noindent
 By Lemma \ref{elemsymcalc}
\begin{align*}
\displaystyle f_{V_H} = 4 \sum_{j = 0}^{\lceil r/2 \rceil - 1} e_{r + 2j+1}(\mathbf{x})e_{r-2j-1}(\mathbf{x}) + N(x_1, x_1, \dots, x_n, x_n).
\end{align*} \noindent
The support of $e_{r + j}(\mathbf{x})e_{r-j}(\mathbf{x})$ is contained in the support of $e_{r+1}(\xx)e_{r-1}(\xx)$ for each $1 \leq j \leq r$. Hence $f_{V_H}$ has the same support  as the polynomial
\begin{align*}
\displaystyle W_{V_H} = 4e_{r+1}(\xx)e_{r-1}(\xx) + \frac{3}{r+1}N(x_1, x_1, \dots, x_n, x_n)
\end{align*} \noindent
which in turn is stable by Theorem \ref{basegenstab}. Hence if we replace $x_i^k$, $k=0,1,2$, in $W_{V_H}$ with $e_k(x_i, x_{i'})/\binom 2 k$, we obtain a polynomial 
which is stable by the Grace-Walsh-Szeg\H{o} theorem, and has the same support as $h_{V_H}$. Hence $V_H$ is a WHPP-matroid so $V_H$ is hyperbolic by Proposition \ref{whpphyp}.
\end{proof}

\begin{proof}[Proof of Theorem \ref{genvamoshpp}]
Recall the notation in the proof of Theorem \ref{genvamoshyp}. If $r=2$, then $W_{V_G} = f_{V_G}$, so that $V_G$ has the half-plane property by the proof of Theorem \ref{genvamoshyp}.
\end{proof}

\end{document}